\documentclass[12pt]{amsart}
\usepackage{amsmath}
\usepackage{amsfonts}
\usepackage{amssymb}
\usepackage{amsthm}
\usepackage[all]{xy}
\usepackage{color}
\usepackage{amscd}
\usepackage[mathscr]{eucal}
\usepackage{verbatim}
\usepackage{tikz-cd}
\usepackage{hyperref}
\newtheorem{theorem}{Theorem}[section]
\newtheorem{lemma}[theorem]{Lemma}
\newtheorem{proposition}[theorem]{Proposition}
\newtheorem{corollary}[theorem]{Corollary}
\newtheorem{definition}[theorem]{Definition}

\theoremstyle{definition} 
\newtheorem{remark}[theorem]{Remark}

\numberwithin{equation}{section}

\title[Projective Poincar\'{e} and Picard bundle] {Projective Poincar\'{e} and Picard bundles for moduli spaces of vector bundles over nodal curves}

\author[C. Arusha]{C. Arusha}
\address{ Indian Institute of Science Education and Research, Bhopal}
\email{arushnmaths@hotmail.com}

\author[Usha N. Bhosle]{Usha N. Bhosle}
\address{Indian Statistical Institute, Bangalore}
\email{usnabh07@gmail.com}

\author[Sanjay Kumar Singh]{Sanjay Kumar Singh}
\address{ Indian Institute of Science Education and Research, Bhopal}
\email{sanjayks@iiserb.ac.in}

\begin{document}
\subjclass[2000]{14H60 (14J60)}

\keywords{Poincar\'{e} Bundle, Picard bundle, Moduli space, Nodal curves, Stability}

\thanks{ The first and third named author would like to thank Prof. Indranil Biswas for some useful discussions.
This work was done during the tenure of the second named author as INSA Senior Scientist at Indian Statistical 
Institute, Bangalore.  A part of this work was done during the Discussion meeting on Moduli of Bundles and Related 
Structures at ICTS, Bangalore. The first and second named author would like to thank ICTS for hospitality and 
stimulating environment. The third named author is supported by the SERB Early Career Research Award (ECR/2016/000649)
by  the Department of Science \& Technology (DST), Government of India}

\date{}

\begin{abstract}
Let $U^{'s}_L(n,d)$ be the moduli space of stable vector bundles of rank $n$ with determinant $L$ where $L$ is a 
fixed line bundle of degree $d$ over a nodal curve $Y$. We prove that the projective Poincare bundle on 
$Y \times U^{'s}_L(n,d)$ and the projective Picard bundle on $U^{'s}_L(n,d)$ are stable for suitable polarisation. 
For a nonsingular point $x \in Y$, we show that the restriction of the projective Poincare bundle
to $\{x\} \times U^{'s}_L(n,d)$ is stable for any polarisation. We prove that for arithmetic genus $g\ge 3$ and for $g=n=2, d$ odd, the Picard group of the moduli space $U'_L(n,d)$ of semistable vector bundles of rank $n$ with determinant $L$ of degree $d$ is isomorphic to $\mathbb{Z}$. 
\end{abstract}

\maketitle
\section{Introduction}
Let $X$ be an irreducible smooth projective curve of genus $g\geq 3$. Let $M_X(n,d)$ denote the moduli space of stable vector bundles of rank $n$ and degree $d$ on $X$. It is well-known that there is a Poincar\'e bundle (universal vector bundle) over $X \times M_X(n, d)$ if and only if $d$ is coprime to $n$ (see \cite{Tyu} for existence in the coprime case, \cite{R} for non-existence in the non-coprime case and \cite{N2} for a topological version of non-existence in the case $d = 0$). If $n$ and $d$ have a common divisor, then there exists a projective Poincar\'e bundle over $X \times M_X(n, d)$. For a fixed line bundle $\xi$ on $X$, let $M_{\xi} \subset M_X(n,d)$ be the subvariety corresponding to vector bundles with determinant $\xi$. The existence and stability of a projective Poincar\'{e} bundle on $X\times M_{\xi}$ were proved in  \cite{BBN}. Our aim in this paper is to prove the existence and stability of a projective Poincar\'{e} bundle when $X$ is a nodal curve (for $n$ and $d$ not necessarily coprime).

         Henceforth, we assume that $Y$ is an integral projective nodal curve of arithmetic genus $g\geq 2$ defined over an algebraically closed field.  Let $p: X \to Y$ be the normalisation map.  Let $U_Y(n,d)$ be the moduli space of semistable torsion-free sheaves of rank $n$ and degree $d$ on the nodal curve $Y$ and $U'_Y(n,d)$ its open dense subvariety corresponding to vector bundles. Denote by $U'_L(n,d)$ the subvariety of $U'_Y(n,d)$ corresponding to vector bundles with determinant $L$ where $L$ is a fixed line bundle. Let $U_L(n,d)$ be its closure in $U_Y(n,d)$. The superfix $^s$ (respectively $^{ss}$) will denote the subset corresponding to stable (respectively semistable) bundles.
         
         If $(n,d)=1$, then there is a Poincar\'e sheaf  on $Y \times U_Y(n, d)$ \cite[Theorem 5.12']{N}.  We show that if $Y$ has geometric genus $g(X) \ge 2$, then for $n$ and $d$ not coprime, there does not exist a Poincar\'e sheaf on  $Y \times V$ for any Zariski open subset of $V \subset U_L(n,d)$ (see Theorem \ref{nonexistencethm}, see also Corollary  \ref{brgppf}). However, there exists a projective Poincar\'e bundle $\mathcal{PU}$ over $Y \times U^{'s}_L(n,d)$ whose restriction to $Y\times \{E\}$ is isomorphic to $P(E)$  for all $E \in U^{'s}_L(n,d)$. Let $x \in Y$ be a nonsingular point. Define $\mathcal{PU}_x = \mathcal{PU}\vert_{x\times U^{'s}_L(n,d)}\, .$ 
  Our first main results are summed up in the following theorem.

\begin{theorem} (Theorem \ref{Poincare})

	Let $Y$ be an integral projective nodal curve of arithmetic genus $g \geq 3$ defined 
over an algebraically closed field. Let $n \geq 2$ be an integer and let $L$ be a 
line bundle of degree $d$ on $Y$. Let $U^{'s}_L(n,d)$ denote the moduli space of stable vector bundles on $Y$ of rank $n$ 
and determinant $L$. Denote by $\mathcal{PU}$ the projective Poincar\'{e} bundle on $Y\times U^{'s}_L(n,d)$ and define 
$\mathcal{PU}_x = \mathcal{PU}\vert_{x\times U^{'s}_L(n,d)}\, ,$ where $x\in Y$ is a nonsingular point.
\begin{enumerate}

\item  The projective Poincar\'e bundle $\mathcal{PU}_x$ on $U^{'s}_L(n,d)$ is stable for any polarisation.  
\item Let $\eta$ and $\theta_L$ be divisors 
 defining the polarisation on $Y$ and $U^{'s}_L(n,d)$ respectively. Then the projective Poincar\'e bundle $\mathcal{PU}$ 
 is stable with respect to $a\eta+b\theta_L$, $a,b>0$.
\end{enumerate}
\end{theorem}

                  If $(n,d)=1$, then for $d\ge 2n(g-1)$, the direct image of a Poincar\'e bundle on $U'_Y(n,d)$ is a vector bundle, 
called a Picard bundle. For $n$ and  $d$ non-coprime, there is no Poincar\'e bundle and hence no Picard bundle. However, 
one can construct a projective Picard bundle on $U^{'s}_L(n,d)$ (following \cite{BBN}). 

\begin{theorem} (Theorem  \ref{Picard})

    Let the assumptions be as in Theorem \ref{Poincare}. Let $d \ge 2n(g -1)$. Then the projective Picard bundle $\mathcal{PW}$ on $U^{'s}_L(n,d)$ is stable.
\end{theorem}

             For our definitions and theorems to be valid for $g \ge 3$, we needed to prove a few results on codimensions of complements of some open subsets of $U_Y(n,d)$ and $U_L(n,d)$. These  results are of independent interest. They have many applications; we give here a few of them.  Some of these results were proved by one of the authors under the assumption that the geometric genus $g(X)$ of $Y$ is at least two; this restriction is removed now.   

\begin{theorem} (Theorem \ref{thmcodims})

	Let $Y$ be an integral nodal curve of arithmetic genus $g$. Assume that $n$ and $d$ are not coprime. 
\begin{enumerate}
\item For $g \ge 2$ and  $n \ge 3$ (resp. $n = 2$),
	$${\rm codim}_{U'_Y(n,d)} (U'_Y(n,d)- U^{'s}_Y(n,d)) \ge 2(g-1) (\ {\rm resp.} \  \ge g-1)\, .$$
\item For $g \ge 2$ and  $n \ge 3$ (resp. $n = 2$),
	$${\rm codim}_{U'_L(n,d)} (U'_L(n,d)- U^{'s}_L(n,d)) \ge 2(g-1) (\ {\rm resp.} \  \ge g-1)\, .$$
\end{enumerate}
\end{theorem}

As an application, we have the following corollary.

\begin{corollary} (Corollary \ref{coro1} )
Let $Y$ be an integral nodal curve of arithmetic genus $g\geq 2$. 
Assume that if $n=2$ and $g =2$ then $d$ is odd. Then 
\begin{enumerate}
\item $\ {\rm Pic} \  U^{'s}_L(n,d) \cong \mathbb{Z}\, .$
\item $\ {\rm Pic} \  U'_L(n,d) \cong \mathbb{Z}\, .$ 
\item $\ {\rm The \ class \ group} \ Cl(U_L(n,d)) \cong \mathbb{Z}\, .$
 ~~~ $ \ {\rm The \ class \ group} \ Cl(U'_L(n,d)) \cong \mathbb{Z}\, .$
 \end{enumerate}
\end{corollary}

        Let $\overline{U}^{'ss}_Y(n,d)$ (respectively $\overline{U}^{'s}_Y(n,d)$) be the subset of $U_Y(n,d)$ corresponding to vector bundles $F$ such that $p^*F$ is semistable (respectively stable). Similarly, $\overline{U}^{'ss}_L(n,d)$ (respectively $\overline{U}^{'s}_L(n,d)$) 
 denotes the subset of $U_L(n,d)$ corresponding to vector bundles $F$ such that $p^*F$ is semistable (respectively stable). 
        
\begin{theorem}(Theorem \ref{thmcodim} ) 
For $n \ge 3$ (resp. $n=2$) and ${\overline{U}_Y^{'ss}(n,d)} \neq {\overline{U}_Y^{'s}(n,d)}$  one has:\\
\begin{tabular}{lll}
(1) & {\rm codim}$_{U'_Y(n,d} (U'_Y(n,d) - \overline{U}^{'ss}_Y(n,d)) \ge 2g(X)-1$  & (resp. $g(X)$). \\
(2) & {\rm codim}$_{U'_L(n,d)} (U'_L(n,d) - \overline{U}^{'ss}_L(n,d)) \ge 2g(X)-1$  &  (resp. $g(X)$).\\
(3) & {\rm codim}$_{U'_Y(n,d)} (U'_Y(n,d) - \overline{U}^{'s}_Y(n,d)) \ge 2g(X)-2$  & (resp. $g(X)-1$). \\
(4) & {\rm codim}$_{U'_L(n,d)} (U'_L(n,d) - \overline{U}^{'s}_L(n,d)) \ge 2g(X)-2$  &  (resp. $g(X)-1$).\\
(5) & {\rm codim}$_{\overline{U}^{'ss}_Y(n,d)}  (\overline{U}^{'ss}_Y(n,d) - \overline{U}^{'s}_Y(n,d)) \ge 2g(X)-2$  & (resp. $g(X)-1$). \\
(6) & {\rm codim}$_{\overline{U}^{'ss}_L(n,d)}  (\overline{U}^{'ss}_L(n,d) - \overline{U}^{'s}_L(n,d)) \ge 2g(X)-2$  & (resp. $g(X)-1$). 
\end{tabular}

\end{theorem}
         
         From this, we deduce the following result.
         
\begin{corollary} 
        Let $Y$ be a complex nodal curve (with at least one node).  For $g(X) \ge 2$, except possibly when $g(X) = n = 2$ and $d$ even,  the moduli space $U'_Y(n,0)$ (respectively  $U'_L(n,0)$) has a big open dense subset (i.e. this open subset has a complement of codimension at least $2$) whose elements correspond to  vector bundles which come from representations of the  fundamental group of $Y$.    
\end{corollary}

For similar results  for $U'_Y(n,d)$ and $U'_L(n,d)$, for any $d$, see Corollary \ref{cororep}.

                The codimension computations are done in Sections 6 and 7. These sections are independent and can be read directly without going through the previous sections. The results on projective Poincar\'e bundles are proved in Section 4, those on projective Picard bundles are proved in Section 5. Constructions and results needed for the proofs of the results in Sections 4 and 5 are contained in the rest of the sections.

\section{Preliminaries}

           We start with some definitions and general results needed in the paper. 

\subsection{Degree of a sheaf on a big open set}\hfill

        Let $W$ be a variety of dimension $m$  with an ample line bundle $H$. 
Let $\mathcal{E}$ denote a torsion-free sheaf on the projective variety $W$. 
Let $C$ denote a general complete intersection curve rationally equivalent to  $H^{m-1}$. 
 If the singular set of $W$ has codimension at least $2$, then the general complete intersection curve 
 can be chosen to lie on the set of nonsingular points of $W$.

We define the degree of $\mathcal{E}$ with respect to the polarisation $H$ as,
\begin{equation*}
deg~ \mathcal{E}=deg ~\mathcal{E}\vert_C.
\end{equation*}

\begin{remark}
Let $U$ be an open subset of $W$ such that codimension of $S=W - U$ in $W$ is at least 2.
Since $\dim ~S \leq \dim W-2$, for any torsion-free coherent sheaf $\mathcal{F}$ on the open subset 
$U$, the direct image $i_*(\mathcal{F})$ is a torsion-free coherent sheaf on $W$.

Proof: Note that the restriction of any non-zero section of $\mathcal{F}$ over  $V\subset W$ on $U\cap V$ is non-zero and $\mathcal{F}$ is torsion-free on $U$.
\end{remark}

\begin{definition}
Let $U$ be an open subset of $W$ such that codimension of $W-U$ in $W$ is at least 2. Let $\mathcal{F}$ be a torsion-free sheaf on $U$. We define $deg(\mathcal{F})= deg (i_*(\mathcal{F}))$, where $i:U\to W$ is the inclusion map and $i_*(\mathcal{F})$ is the direct image sheaf of $\mathcal{F}$ on $W$.

\end{definition}

\subsection{ Stability of Projective Bundles}\hfill

Let $U$ be an open subset of a projective variety $W$ such that the codimension of the complement of $U$ is at least $2$.
Let $P$ be a projective bundle on $U$ and let $P'$ be a projective subbundle of the 
restriction of $P$ to a Zariski open subset $Z$ of $U$ whose complement has codimension at least $2$.

Let $p$ and $p'$ denote the projections of $P$ and $P'$ to $U$ and $Z$, respectively. 
As $P'$ and $P\vert_Z$ are smooth over $Z$ and $i: P' \to P\vert_Z $ is a closed embedding with 
ideal sheaf, say $\mathcal{I}$, by \cite[Theorem C.15, D.2.7]{ES} we get an exact sequence of vector bundles on $P'$ 

\begin{equation*}
	0\longrightarrow\mathcal{I}/\mathcal{I}^2\longrightarrow\Omega_{P\vert_{Z}/Z}\otimes\mathcal{O}_{P'}
	\longrightarrow\Omega_{P'/Z}\longrightarrow0.
\end{equation*}
Dualising it, we get,
\begin{center}
	$0\rightarrow\mathcal{T}_{P'/Z}\rightarrow\mathcal{T}_{P|_{Z}/Z}\rightarrow N_{P'/P}\rightarrow0$
\end{center}
where $\mathcal{T}_{P'/Z}, \mathcal{T}_{P|_{Z}/Z}$ are the relative tangent bundles (with respect to the maps $p': P'\to Z$ and $p: P|_Z\to Z$ respectively) and $N_{P'/P}$ is the normal bundle of $P'$ in $P|_Z$. 

Let $N_1=N_{P'/P}$ and $N = {p'}_*(N_1)$.  
Using a theorem of Grauert (\cite[Corollary 12.9]{Ha}) and the fact that $H^i
(\mathbb{P}^n,T\mathbb{P}^n)= 0$ for all $i\geq 1$, we get $N$ is a vector bundle on $Z$.

\begin{definition}\label{defn}
	The projective bundle $P$ is stable(semistable) if for every subbundle $P'$, 
	$$\deg N>0 ~~~(\deg N\geq0)$$
\end{definition}

\begin{remark} \label{rmk1}
(1) Definition \ref{defn} is equivalent to the standard definition of stability for a principal
$PGL(n)$-bundle(Remark 2.2, \cite{BBN}). Also note that if $P=P(E)$ then $P' = P(F)$, where $F$ is a subbundle of $E|_Z$ such that $E|_Z/F$ is locally free on $Z$. In this case, it is easy to verify that the stability of $E$ is same as the stability of $P(E)$.
Note that the above definition is for a variety of dimension at least $2$. \\
(2) There is also a notion of stability for a projective bundle on a curve which is similar to the above definition. For any 
integral curve $Y$, all projective bundles are of  the form $P(E)$ for some vector bundle $E$ (since the Brauer Group $Br(Y)=0$).
We remark that the stability of a projective bundle $P(E)$ and the Mumford stability of $E$ are same on a smooth curve, but it is not the case for a singular curve.
\end{remark}

\subsection{Torsion-free sheaves on nodal curves}\hfill

Let $Y$ be an integral projective curve with only $m$ nodes (ordinary double points) as 
singularities defined over an algebraically closed field $k$. 
Let $p: X\rightarrow Y$ be the  normalisation map. Let $g = h^1(Y, {\mathcal O}_Y)$ be the arithmetic genus of  $Y$; 
we assume that $g\geq 2$. 
Let $g(X) = h^1(X, \mathcal{O}_X)$ be the geometric genus of $Y$. 

\noindent For a torsion-free sheaf $\mathcal{E}$, let ${\rm rk}(\mathcal{E})$ denote the generic rank of $\mathcal{E}$ and $\deg(\mathcal{E})=\chi(\mathcal{E})-{\rm rk}(\mathcal{E})\chi(\mathcal{O}_Y)$ denote the degree of $\mathcal{E}$. Let $\mu(\mathcal{E})=\deg(\mathcal{E})/{\rm rk}(\mathcal{E})$ denote the slope of $\mathcal{E}$.

We say that $\mathcal{E}$ is stable (respectively semistable) if for every coherent subsheaf $\mathcal{F}$ of $\mathcal{E}$ on $Y$,
 $$\mu(\mathcal{F}) < (\leq)  \ \mu(\mathcal{E}) .$$
 This is equivalent to the following:
 for every torsion-free quotient $\mathcal{G}$ of $\mathcal{E}$ on $Y$,
 $$\mu(\mathcal{E}) <  (\leq) \ \mu(\mathcal{G}) .$$
 
         The stalk $\mathcal{E}_{(y_j)}$ of a torsion-free sheaf $\mathcal{E}$ at a node $y_j$ is isomorphic to $\mathcal{O}_{y_j}^{\oplus a_j} \oplus m_{y_j}^{\oplus b_j}$, where $\mathcal{O}_{y_j}$ is the local ring at $y_j$ and $m_{y_j}$ is its maximum ideal, 
 $a_j, b_j$ are integers with $a_j + b_j= {\rm rk}(\mathcal{E})$. 
Then $b_j$ is called the {\it local type} of $\mathcal{E}$ at $y_j$.  

\begin{remark}
Let  $U_Y(n,d)$ denote  the moduli space of $S$-equivalence classes of semistable 
torsion-free sheaves  of rank $n$ and degree $d$ on $Y$ and $U'_Y(n,d) $ its open dense 
subvariety corresponding to locally free sheaves (vector bundles). For a line bundle $L$ over $Y$, denote  by 
$U'_L(n,d)$  the closed subset of $U'_Y(n,d)$ corresponding to vector bundles with fixed determinant $L$,
and $U_L(n,d)$  the closure of $U'_L(n,d)$ inside $U_Y(n,d)$ with a reduced structure.

   The moduli space $U_Y(n,d)$ is an irreducible seminormal projective variety (\cite{Re}, \cite[Theorem 4.2]{Su}). $U'_Y(n,d)$ and $U'_L(n,d)$ are normal quasi-projective varieties (\cite{N}). They are known to be  locally factorial for $g(X) \ge 2$ \cite[Theorem 3A]{UB3}. We shall soon see that $U'_L(n,d)$ is locally factorial for $g \ge 2$ (except possibly for $g=2=n, d$ even).   
 The open subset $U^{'s}_L(n,d) \subset U'_L(n,d)$ corresponding to stable vector bundles is nonsingular. 
 
          There is a canonically defined ample line bundle $\theta_U$ on $U_Y(n,d)$, the determinant of cohomology line bundle. 
 Let $\theta_L \to U_L(n,d)$ be its restriction. Assume that $g\ge 2$ and if $g=2=n$ then $d$ is odd. Then we shall prove that 
 Pic $U^{'s}_L(n,d) \cong \mathbb{Z}$ (Corollary \ref{coro1}; \cite[Theorem I]{UB4}, for $g(X) \geq 2$). The  restriction of $\theta_L$ to $U^{'s}_L(n,d)$ is the generator of Pic $U^{'s}_L(n,d)$.
\end{remark}

We start with some results on the codimensions of subsets of moduli spaces needed in the nodal case.

\subsection{ Codimensions of subsets of moduli spaces} \label{codimensions}

\begin{theorem}\label{thmcodim'}  \cite[Theorem 2.5]{UB1}\\
Let $Y$ be an integral nodal curve of arithmetic genus $g\geq 2$ with $m$ nodes, $m\geq 1$. For $n \ge 2$, 
the codimension of $U_L(n,d)- U'_L(n,d)$ in $U_L(n,d)$  is more than $2$.
\end{theorem}

\begin{theorem} (Theorem \ref{thmcodims})
	Let $Y$ be an integral nodal curve of arithmetic genus $g$. 
\begin{enumerate}
\item For $g \ge 2$ and  $n \ge 3$ (resp. $n = 2$),
	$${\rm codim}_{U'_Y(n,d)} (U'_Y(n,d)- U^{'s}_Y(n,d)) \ge 2(g-1) (\ {\rm resp.} \  \ge g-1)\, .$$
\item For $g \ge 2$ and  $n \ge 3$ (resp. $n = 2$),
	$${\rm codim}_{U'_L(n,d)} (U'_L(n,d)- U^{'s}_L(n,d)) \ge 2(g-1) (\ {\rm resp.} \  \ge g-1)\, .$$
\end{enumerate}
\end{theorem}

This theorem and its corollary will be proved in Section \ref{codimstable}. The proofs in Section \ref{codimstable} are  independent of the results in the sections before it. 

\begin{corollary} \label{coro1}
Let $Y$ be an integral nodal curve of arithmetic genus $g\geq 2$ with $m$ nodes. 
Assume that $g \ge 2$ and if $n=2$ and $g =2$ then $d$ is odd. Then 
\begin{enumerate}
\item $\ {\rm Pic} \  U^{'s}_L(n,d) \cong \mathbb{Z}\, .$
\item $\ {\rm Pic} \  U'_L(n,d) \cong \mathbb{Z}\, .$ 
\item $\ {\rm The \ class \ group} \ Cl(U_L(n,d)) \cong \mathbb{Z}\, .$
 ~~~ $ \ {\rm The \ class \ group} \ Cl(U'_L(n,d)) \cong \mathbb{Z}\, .$
 \end{enumerate}
\end{corollary}

\begin{remark} \label{localfacto}
(1) Let the assumptions be as in Corollary \ref{coro1}. Then as in \cite[Lemma 2.3]{UB0}, we can show that 
$U'_L(n,d)$ and $\tilde{U}_L(n,d)$ are locally factorial and 
$${\rm Pic} \ \widetilde{U}_L(n,d) \ \cong \ \mathbb{Z}\, .$$      
(2) For $g=2, n=2$ and $d$ even, {\rm codim}$_{U'_L(2,d)} (U'_L(2,d) - U^{'s}_L(2,d)) =1$. 
The variety $U_L(2,d) = U'_L(2,d) \cong \mathbb{P}^3$ has Picard group isomorphic to $\mathbb{Z}$. 
One has ${\rm Pic} \ U^{'s}_L(2,d) \cong \mathbb{Z}/ 4\mathbb{Z}$ \cite[Proposition 4.3]{UB'}.
\end{remark}

\subsection{Non-existence of a universal vector bundle on $Y \times U'_L(n,d)$}\hfill

From the exact sequence of group schemes
\begin{equation*}
1\rightarrow  G_m \rightarrow GL(n)\rightarrow PGL(n)\rightarrow 1,
\end{equation*}
we have the following exact sequence of cohomology (non-abelian) groups
\begin{equation*}
H^1(X_{\acute{e}t}, \mathbb{G}L(n))\overset{\mathbb{P}}{\rightarrow} H^1(X_{\acute{e}t}, \mathbb{P}GL(n))\overset{\delta_n}{\rightarrow} H^2(X_{\acute{e}t},\mathbb{G}_m).
\end{equation*}

The natural map $H^1(X_{\acute{e}t}, \mathbb{G}L(n))\overset{\mathbb{P}}{\rightarrow} H^1(X_{\acute{e}t}, \mathbb{P}GL(n))$ corresponds to projectivisation. The Brauer group captures the projective bundles that do not arise from vector bundles.  
      For (quasi) projective varieties, the Brauer group is the same as the cohomological Brauer group $H^2(X_{\acute{e}t},\mathbb{G}_m)$ (\cite{Ga}).

One way to see that there is no Poincar\'{e} bundle on $U_L'(n,d)$ is to exhibit a Brauer obstruction to the existence of such a bundle. Let ${\rm Br}(U_L'(n,d))$ denote the cohomological Brauer group $H^2_{\acute{e}t}(U_L'(n,d),\mathbb{G}_m)$. We have the following result from \cite{UBB} where the Brauer group and other properties of this moduli space have been studied.

\begin{theorem} \label{ubb}
	Assume that $g(X)\geq 2$; also, if $g(X)=2=n$, assume that $d$ is odd. Then the Brauer group $Br(U_L'(n,d))\cong \mathbb{Z}/h\mathbb{Z}$ where $h=gcd(n,d)$ and the Brauer group is generated by the Brauer class of $\mathcal{PU}_x$, the restriction of the projective Poincar\'e bundle to $U'_L(n,d)\times x$, where $x \in Y$ is a nonsingular point.  
\end{theorem}

\begin{corollary} \label{brgppf}
	Assume that  $g(X) \ge 2$ and if $g(X) =2$, then $n \ge 3$. Then for $n$ and $d$ non-coprime, there is no Poincar\'{e} bundle on $Y\times U'_L(n,d)$.
\end{corollary}

\begin{proof}
	If there exists a Poincar\'{e} bundle $E$ on $U_L'(n,d)\times Y$, then by the uniqueness of projective Poincar\'{e} bundles, we have $\mathcal{PU}\cong P(E)$ and $\mathcal{PU}_x\cong P(E_x)$ for a nonsingular point $x\in Y$. This will imply that the Brauer class of $\mathcal{PU}_x$ is trivial as it is the projectivisation of a vector bundle. As a consequence we get $Br(U_L'(n,d))=\{0\}$ contradicting Theorem \ref{ubb}.
\end{proof}

       In Section \ref{nonexist}, we shall prove the more general result that for $n$ and $d$ non-coprime, there is no Poincar\'e bundle on $Y\times V$ where $V \subset U'_L(n,d)$ is any Zariski open subset.

\section{The Morphism $\psi_{F,x}$} \label{psiFx}

\subsection{Elementary Transformations} 
For a vector bundle $F$ of rank $n$ on a projective algebraic curve $X$, the elementary operation of degree $k$ is defined as follows:
Let $x\in X$ be a regular point of $X$. Consider a $k$-dimensional subspace $w\subset F_x$ in the fibre of $F$ over the point $x$.  
This uniquely determines a surjective homomorphism
\begin{equation*} 
F\overset{\alpha(w)}\longrightarrow k_x^{n-k}\longrightarrow 0,
\end{equation*}
where $k_x^{n-k}$ is the skyscraper sheaf of dimension $n-k$ supported at $x\in X$
by requiring the condition that $\ker \alpha(w)_x=w$. 
Here $\alpha(w)_x$ denotes the restriction of $\alpha(w)$ to the fibre over $x$. 
Then the kernel of $\alpha(w)$ is a locally free sheaf and is usually denoted by $\text{elm}^k_x(w)(F)$.
Thus we obtain the following exact sequence:
\begin{equation*}
0\longrightarrow \text{elm}^k_x(w)(F)\longrightarrow F \overset{\alpha(w)}\longrightarrow k_x^{n-k}\longrightarrow 0
\end{equation*}
and the operation $F\mapsto \text{elm}^k_x(w)(F)$ is called an elementary operation of degree $k$ (\cite{MM}).

Let $x\in Y$ be a nonsingular point of $Y$, and $E$ a vector bundle over $Y$ of rank $n$ and degree 
$d$ with $\det E=L$. Let $\ell\subset E_x$ be a line in the fiber of $E$ at $x$. We obtain a vector bundle $F$ defined by the short exact sequence
\begin{equation}\label{El}
0\rightarrow F(-x)\rightarrow E\rightarrow E_x/\ell\rightarrow0.
\end{equation}

Now let $F$ be a vector bundle of rank $n$ on the nodal curve $Y$  such that $\det F \cong L(x)$, $L$ is a line bundle of degree $d$. 
Let $\mathbf{P}:=P(F_x^*)$ be the projective space parametrising the hyperplanes in the fibre $F_x$. Denote by $p:Y\times P(F_x^*) \rightarrow Y$ the projection onto $Y$ and by $i: P(F_x^*)\rightarrow Y\times P(F_x^*)$ the inclusion map defined by $\phi\mapsto(x,\phi)$.

Consider a nontrivial $\phi \in P(F_x^*)$, i.e., a nontrivial homomorphism $\phi: F\rightarrow k_x$. We have an exact sequence 
\begin{equation}\label{Ephi}
0 \longrightarrow E \longrightarrow F \overset{\phi}{\longrightarrow}k_x\longrightarrow 0,
\end{equation}
where $E = E_{\phi}:= \ker\phi$. Then $E$ is a vector bundle since $x$ is a nonsingular point.

\noindent We can reconstruct $E$ and $F$ from each other from the following commutative diagram. 

\begin{center}
	\begin{eqnarray}\label{CD}
	\begin{tikzcd}
	&  0\arrow{d} &  0\arrow{d} & & \\
	& F(-x)\arrow{d}\arrow[equal]{r}& F(-x)\arrow{d}& & \\
	0\arrow{r}& E\arrow{r}\arrow{d} & F\arrow{r}{\phi}\arrow{d} & k_x\arrow[equal]{d}\arrow{r} & 0\\
	0\arrow{r} & E_x/\ell\arrow{r}\arrow{d}& F_x\arrow{r}\arrow{d} & k_x\arrow{r}& 0\\
	& 0 &  0 & & 
	\end{tikzcd}
	\end{eqnarray}
\end{center}

 For every 
\begin{equation*}
\phi\in \text{Hom}(F,k_x)=\text{Hom}(F_x,k_x)= F_x^*
\end{equation*}
we get a vector bundle $E_{\phi}$ and as $\phi$ varies over ${P}(F_x^*)$ we obtain a family of vector bundles parametrised by 
${P}(F_x^*)$, say $\mathcal{E}$, and we have the following exact sequence
\begin{equation}\label{SESfamily}
0\rightarrow\mathcal{E}\rightarrow p^*F\rightarrow i_*\mathcal{O}_{\{x\}\times\textbf{P}}(1)\rightarrow 0
\end{equation}
 on $Y\times \textbf{P}$, where $\mathcal{E}$ is a vector bundle.

\subsection{$(l,m)$-stability of a coherent sheaf on $Y$ and the morphism $\psi_{F,x}$}
 
\begin{definition}
	Let $l$ and $m$ be integers. A torsion-free sheaf $\mathcal{E}$ on $Y$ is said to 
	be $(l,m)$-stable if for every proper subsheaf $\mathcal{G}$ of $\mathcal{E}$ with a 
	torsion-free quotient
	\begin{equation*}
	\frac{\deg(\mathcal{G})+l}{{\rm rk}(\mathcal{G})}<\frac{\deg(\mathcal{E})+l-m}{{\rm rk}(\mathcal{E})}.
	\end{equation*}
	
\end{definition}
\begin{remark}
(1)  A torsion-free sheaf $\mathcal{E}$ is stable if and only if it is $(0,0)$-stable.\\
(2)  A $(0,1)$-stable torsion-free sheaf is stable.
\end{remark}

\begin{lemma}
	If the bundle $F$ in the sequence ~\eqref{Ephi} is $(0,1)$-stable, then $E_{\phi}$ is stable.
\end{lemma}

\begin{proof}
	The proof is similar to the case when $Y$ is smooth [\cite{NR}, Lemma 5.5]. If $E_{\phi}$ is not stable, then there exists  a proper subsheaf $G$ of $E_{\phi}$ with a torsion-free quotient 
and satisfying $\mu(G)> \mu(E_{\phi})$. If we denote by $G'$ the maximal torsion-free subsheaf   
generated by the image of $G$ in $F$, then 
	\begin{equation*}
	\frac{\deg G'+0}{{\rm rk}(G')}  \ge \mu(G)>\mu(E_{\phi})=\frac{\deg F-1}{{\rm rk}(F)}
	\end{equation*}
	contradicting the $(0,1)$-stability of $F$.
\end{proof}

By the above lemma,  for a $(0,1)$-stable vector bundle $F \in U_{L(x)}'(n,d+1)$ we get the vector bundle $\mathcal{E}$ in \eqref{SESfamily}, i.e.,  a family of stable vector bundles of rank $n$ and determinant $L$ on $Y$. And using the universal property of $U_L'(n,d)$, we get a morphism, 
\begin{center}
	$\psi_{F,x}: {P}(F_x^*)\rightarrow U'_L(n,d)$.
\end{center}

\begin{lemma}\label{psi}
	$\psi_{F,x}$ is an isomorphism onto its image.
\end{lemma}

\begin{proof}
This can be proved exactly as \cite[Lemma 5.9]{NR} (see \cite[Lemma 5.6]{NR} and \cite[Lemma 3]{BBGN} for injectivity). 
We note that $\psi_{F,x}$ maps into $U^{'s} _L(n,d)$.
\end{proof}

The diagrams \ref{CD} obtained by varying $\phi \in P(F_x^*)$ combine to form the following diagram:

\begin{center}
	\begin{eqnarray}\label{Diagram1}
	\begin{tikzcd}
	&  0\arrow{d} &  0\arrow{d} & & \\
	& p_1^*F(-x)\arrow{d}\arrow[equal]{r}& p_1^*F(-x)\arrow{d}& & \\
	0\arrow{r}& \mathcal{E}\arrow{r}\arrow{d} & p_1^*F\arrow{r}\arrow{d} & i_*\mathcal{O}_{\mathbf{P}}(1)\arrow[equal]{d}\arrow{r} & 0\\
	0\arrow{r} & i_*\Omega^1_{\mathbf{P}}(1)\arrow{r}\arrow{d}& F_x\otimes_{\mathbb{C}}i_*\mathcal{O}_{\mathbf{P}}\arrow{r}\arrow{d} & i_*\mathcal{O}_{\mathbf{P}}(1)\arrow{r}& 0\\
	& 0 &  0 & & 
	\end{tikzcd}
	\end{eqnarray}
\end{center}

\begin{lemma}[\cite{BBGN}, Lemma 3.1]
Let $\mathcal{E}_{x} := \mathcal{E}\vert_{x \times P(F_x^*)}$. 
	There is an exact sequence of vector bundles
	\begin{equation}\label{exact}
	0\rightarrow\mathcal{O}_{\mathbf{P}}(1)\rightarrow \mathcal{E}_{x}\rightarrow\Omega^1_{\mathbf{P}}(1)\rightarrow 0
	\end{equation}
	on $P(F_x^*)$.
\end{lemma}

\begin{lemma}\label{subbundle}
	Let $W \subset \mathcal{E}_x$ be a non-zero coherent subsheaf of the vector bundle $\mathcal{E}_x$ in \eqref{exact} such that:
	\begin{itemize}
		\item the quotient $\mathcal{E}_x/W$ is torsion-free, and
		\item $\deg W$/rk $W\geq \deg (\mathcal{E}_x)/$rk $(\mathcal{E}_x)$.
	\end{itemize}
	Then $W$ contains the line subbundle $\mathcal{O}_{\mathbf{P}}(1)$ of $\mathcal{E}_x$ defined in \eqref{exact}.
\end{lemma}

\begin{proof}
	By considering Harder-Narasimhan filtration of $W$,
	we can choose a sub-sheaf $W_1$  of $W$   such that $W_1$ is stable, $W/W_1$ is torsion-free, and 
	\begin{equation*}
	\frac{\deg W_1}{\text{rk }W_1}\geq\frac{\deg W}{\text{rk }W}.
	\end{equation*}
	It follows from the exact sequence \eqref{exact} that $\deg \mathcal{E}_x=0$ and we have the 
	following inequality:
	\begin{equation*}
	\frac{\deg W_1}{\text{rk }W_1}\geq\frac{\deg \mathcal{E}_x}{\text{rk }\mathcal{E}_x}=0>\frac{-1}{n-1}=\frac{\deg \Omega^1_{\mathbf{P}}(1)}{\text{rk }\Omega^1_{\mathbf{P}}(1)}.
	\end{equation*}
	(Note that $\dim \mathbf{P}=\dim P(F_x^*)=n-1$.)	
	
	The vector bundle $\Omega^1_{\mathbf{P}}(1)$ is stable (see \cite{OSS}, Chapter II, Theorem 1.3.2). 
	Since $W_1$ is semistable and $\frac{\deg W_1}{\text{rk }W_1}>\frac{\deg \Omega^1_{\mathbf{P}}(1)}
	{\text{rk }\Omega^1_{\mathbf{P}}(1)}$, there is no non-zero homomorphism from 
	$W_1\rightarrow\Omega^1_{\mathbf{P}}(1)$ 
	(The image of such a morphism will contradict the stability of $\Omega^1_{\mathbf{P}}(1)$). 
	This implies that $W_1\subseteq \mathcal{O}_{\mathbf{P}}(1)$ so that rk$(W_1) =1$ and $\mathcal{O}_{\mathbf{P}}(1)/W_1$ 
	is a torsion sheaf if non-zero.  
	Also, we have the exact sequence
	$$ 0 \to W/W_1\to \mathcal{E}_x/W_1 \to \mathcal{E}_x/W \to 0.$$
	Since $W/W_1$ and $\mathcal{E}_x/W$ are torsion-free so is $\mathcal{E}_x/W_1$. 
	Then $\mathcal{O}_{\mathbf{P}}(1)/W_1$ is a subsheaf of $\mathcal{E}_x/W_1$ implies that $\mathcal{O}_{\mathbf{P}}(1)/W_1$ 
	is torsion-free. It follows that $W_1 = \mathcal{O}_{\mathbf{P}}(1)$.
\end{proof}

\begin{lemma}\label{01Open}
	The (0,1)-stable vector bundles on $Y$, of rank $n$ and determinant $L'$ of degree $d'$ form a non-empty open subset of the moduli space of stable torsion-free sheaves denoted by $U^s_{L'}(n,d')$.
\end{lemma}
\begin{proof}
	We briefly sketch the proof given in \cite{UB1}. Let $C$ be the complement in $U'_{L'}(n,d')$ of the set of (0,1)-stable vector bundles. A vector bundle $F\in C$ if and only if it has a torsion-free subsheaf $G$ of rank $r$ and degree $e$ satisfying $rd'>ne\geq r(d'-1)$. 
	This implies that the ranks and degrees of quotients of $F\in C$ are bounded. The closedness of $C$ follows from the properness of quot schemes. 
	
	Next we estimate the dimension of $C$ and show that it is strictly less than the dimension of $U_{L'}(n,d')$. One can see that to estimate the dimension of $C$ we need to estimate the dimension of the space of extensions   of $H$ by $G$, where $G$ is as above and $H$ is of rank $n-r$  and degree $d'- e$. 
	
	We may assume that $G$ and $H$ are stable, so that $h^0(Hom(H,G))=0$.
	
	\noindent Using the Riemann-Roch Theorem we have
	\begin{center}
		$\chi(F)=\deg(F)+(1-g){\rm rk}(F)$.
	\end{center} 
	From \cite[Lemma 2.5(B)]{UB2}, we have
	\begin{align*}
	\dim Ext^1(H,G) &= \dim H^1(Y,Hom(H,G))+ 2\sum_j b_j(G)b_j(H)\\
	&= (g-1)r(Hom(H,G))-d(Hom(H,G))+2\sum_j b_j(G)b_j(H)\\
	&= rd'-ne+ r(n-r)(g-1)+\sum_j b_j(G)b_j(H).
	\end{align*}
	
	\noindent Hence the dimension $\delta$ of the space of bundles of rank $n$ and degree $d'$ with a fixed determinant $L'$ obtained as an extension of $H$ by $G$ is
	\begin{small}
		\begin{align*}
		\delta 
		&\leq (r^2(g-1)+1-\sum_j b_j(G)^2)+((n-r)^2(g-1)+1-\sum_j b_j(H)^2)-g+ \dim Ext^1(H,G)-1	\\
		&=(g-1)(r^2+(n-r)^2+r(n-r)-1)+rd'-ne-\sum_j (b_j(G)-b_j(H))^2-\sum_j b_j(G)b_j(H)\\
		&\leq (g-1)(r^2+n^2-rn-1)+rd'-ne.
		\end{align*}
	\end{small}	
	\noindent For $C$ to be a proper closed subset, it suffices that the last expression is strictly less than dim $U_{L'}(n,d')=(n^2-1)(g-1)$ 
which is true when $rd'-ne<r(n-r)(g-1)$ and that indeed is the case here since $g\geq3$.
\end{proof}

\noindent We remark that the Lemma holds for any $n$ and $d'$ and the condition in \cite{UB1} that $(n,d'-1)=1$ is not required since we are considering curves with genus $g \geq 3$.

\section{Projective Poincar\'{e} Bundles}

  If $(n,d) = 1$, there is a Poincar\'{e} bundle on $Y\times U'_Y(n,d)$, unique up to tensoring by a line bundle on 
  $U'_Y(n,d)$. If $(n,d)\neq 1$, there is no Poincar\'{e} bundle on $Y\times U'_Y(n,d)$ 
(Corollary \ref{brgppf}, Theorem \ref{nonexistencethm}). However, we can define a 
projective Poincar\'{e} Bundle on $U_Y^{\prime s} (n,d)\times Y$ for any $n$ and $d$ with $n\ge 2$. 
Our construction is very similar to the construction in \cite{BBN}. 

     Grothendieck has proved that, for a positive integer $p$, the torsion-free quotients of 
$\mathcal{O}_Y^p$ which have a fixed Hilbert polynomial $P$ can be parametrised by the points of a projective algebraic scheme $Q$ (Quot Scheme). 
Moreover, there exists a universal quotient coherent sheaf $\mathcal{U}$ over $Y\times Q$. The group
$GL(p)$ acts on $Q$ as the group of automorphisms of $\mathcal{O}_Y^p$ and also on 
the sheaf $\mathcal{U}$. Furthermore, the action of $GL(p)$ 
on $Q$ goes down to an action of $PGL(p)$ on $Q$. However, this is not true for the action on $\mathcal{U}$,  
 the scalar matrix $\lambda(Id)$ acts on $\mathcal{U}$ by scalar multiplication by $\lambda$.

      $U_Y(n,d)$ can be seen as the geometric invariant theoretic (GIT) quotient of a Zariski open set $R_Y$ of $Q$ 
  by the action of $PGL(p)$. In general, $R_Y$ need not be irreducible or nonsingular, 
but the  subset $R'$ of $R_Y$ consisting of those $q$ for which $\mathcal{U}_q$ is locally 
free is irreducible, nonsingular, $PGL(p)$ invariant and open in $Q$ (\cite[Chapter-5]{N}).
Thus $U^{'s}_Y(n,d)$ can be realised as a quotient $\pi :  R^{'s} \to U^{'s}_Y(n,d)$. The universal quotient $\mathcal{U}$ restricts to a vector bundle $\mathcal{U}_{R^{'s}}$ on $Y\times R^{'s}$, 
such that $\mathcal{U}_{R^{'s}}\vert_{Y\times\{r\}}$ is the stable bundle corresponding to $\pi(r)$ for 
all $r\in R^{'s}$. The isotropy subgroup of $GL(p)$ at $r \in R^{'s}$ is Aut $\mathcal{U}_r \cong $ scalar matrices. 
It acts on $\mathcal{U}_r$ with $\lambda Id$ acting by 
multiplication by $\lambda$. Hence $PGL(p)$ acts on the associated projective 
bundle $P(\mathcal{U}_{R^{'s}})$. The quotient $\mathcal{PU}:= P(\mathcal{U}_{R^{'s}})/PGL(p)$ 
is then a projective bundle on $Y\times U^{'s}_Y(n,d)$ whose restriction to $Y\times\{E\}$ 
is isomorphic to $P(E)$ for all $E\in U^{'s}_Y(n,d)$.

\subsection*{The uniqueness of Projective Poincare Bundles}  
\begin{lemma}\label{uniqlem}
Let $\mathcal{E}$ be a vector bundle on $Y\times Z$ such that the restriction of $\mathcal{E}$ to $Y\times\{z\}$ is stable of rank $n$ and determinant $L$ for all $z\in Z$, and let $\psi_{\mathcal{E}}: Z\rightarrow U_Y'(n,d)$ be the corresponding morphism. Then the projective bundles $P(\mathcal{E})$ and $(id_Y\times\psi_{\mathcal{E}})^*(\mathcal{PU})$ are isomorphic. 
\end{lemma}
\begin{proof}
     This can be proved exactly as \cite[Proposition 2.3]{BBN}.
\end{proof}

\begin{corollary}\label{uniqcor}
Suppose that $\pi': R^{'s}\rightarrow U^{'s}_Y(n,d)$ defines $U^{'s}_Y(n,d)$ as a quotient of $R^{'s}$ by a free action of $PGL(M')$ and that:
\begin{enumerate}
\item $\mathcal{E}_{R^{'s}}$ is a vector bundle on $Y\times R^{'s}$ such that $\mathcal{E}_{R^{'s}}|_{X\times\{r'\}}$ is the stable bundle $\pi'(r')$ for all $r'\in R^{'s}$;
\item the action of $PGL(M')$ lifts to $P(\mathcal{E}_{R^{'s}})$.
\end{enumerate}
Then $P(\mathcal{E}_{R^{'s}})/PGL(M')\cong\mathcal{PU}$.
\end{corollary}
\begin{proof}
Apply Lemma \ref{uniqlem} to by taking $Z=R^{'s}$ and $\mathcal{E}=\mathcal{E}_{R^{'s}}$.
\end{proof}

\begin{definition}
 $\mathcal{PU}$ is a projective bundle whose restriction to $Y\times\{E\}$ is isomorphic to $P(E)$ 
 for all $E\in U^{'s}_Y(n,d)$, and we call it the projective Poincar\'{e} bundle. 
 
 We call the restriction of  $\mathcal{PU}$ on  $Y\times U^{'s}_L(n,d)$ 
  the projective  Poincar\'{e} bundle over $Y\times U^{'s}_L(n,d)$ and denote it again by $\mathcal{PU}$.

\end{definition}

Note that by Theorem \ref{thmcodim'} and Theorem \ref{thmcodims}, the codimension of $U_L(n,d) - U^{'s}_L(n,d)$ is at least $2$. So we can define the notion
of degree on $U^{'s}_L(n,d)$ and we can talk about the stability of any projective bundle on  $U^{'s}_L(n,d)$. Similarly
$Y\times U^{'s}_L(n,d) \hookrightarrow Y\times U_L(n,d)$ via $Id\times i$ (where $i: U^{'s}_L(n,d) \hookrightarrow U_L(n,d)$ is the inclusion) as an open set of the projective variety $Y\times U_L(n,d)$ whose complement has codimension at least $2$. So we can talk about the degree and stability of a projective bundle on $Y\times U^{'s}_L(n,d)$.

In the next section, we will study the stability of  $\mathcal{PU}$ and  
$\mathcal{PU}_x:= \mathcal{PU}\vert_{x \times U^{'s}_L(n,d)}$ for $x \in Y_{reg}$.

\subsection{Stability of the restriction of projective Poincar\'{e} bundle to regular points on $Y$} \hfill

	Each point of $\mathcal{PU}_x$ corresponds to a pair $(E,\ell)$ where $E\in U^{'s}_L(n,d)$ 
	and $\ell \in{P}(E_x)$. Define,
	$$H_x= \{(E,\ell)\in \mathcal{PU}_x : F ~~~\mbox{defined in \eqref{El} is (0,1)-stable}\}.$$ 
	
	Let $p: H_x\rightarrow U^{'s}_L(n,d)$ and $q: H_x\rightarrow U'_{L(x)}(n,d+1)$ be defined 
	by $p(E,\ell)=E$ and $q(E,\ell)=F$ (defined in \eqref{El}) respectively.
	
\begin{lemma} \label{L3.5BBN}
	\begin{enumerate}
		\item The set $H_x$ is non-empty and Zariski open in $\mathcal{PU}_x$ and $q$ is a morphism.
		
		\item $p(H_x)$ is non-empty and Zariski open in $U^{'s}_L(n,d)$.
	\end{enumerate}
\end{lemma}

\begin{proof} The proof is essentially the same as in \cite[Lemma 3.5]{BBN},  we give an outline of it.
We can construct $U^{'s}_L(n,d)$ as a quotient $\pi: R^{'s} \rightarrow U'_L(n,d)$ and for a fixed regular point $x\in Y$ we have the bundle $P(\mathcal{U}_{R^{'s}})_x$ on $R^{'s}$. 
A point in ${P(\mathcal{U}_{R^{'s}})}_x$ corresponds to a pair $(E,\ell)$, where $E\in R'^s$ and $\ell \in P(E_x)$. 
Now corresponding to each such pair with $E$ in a Zariski open subset of $R^{'s}$,  we obtain an element $F\in U'_{L(x)}(n,d+1)$ from the exact sequence \eqref{El}. Thus a $PGL(p)$-invariant  open subset $V'$ of $P(\mathcal{U}_{R^{'s}})_x$ parametrises a family of elements of $U'_{L(x)}(n,d+1)$ and by the universal property of the moduli space, we get a morphism $q': V' \rightarrow U'_{L(x)}(n,d+1)$ which goes down to a morphism $q: V \to U'_{L(x)}(n,d+1)$. Let $H_x$ be the subset of $\mathcal{PU}_x$ for which $F$ corresponding to $(E,\ell)$ is $(0,1)$-stable. Then by Lemma \ref{01Open}, $H_x$ is non-empty and Zariski open in $\mathcal{PU}_x$. \\
By Part ($1$), $p(H_x)$ is non-empty. It is Zariski open since $p$ is the restriction of the projection morphism 
$P(\mathcal{U}_x) \to U^{'s}_L(n,d)$. 
\end{proof}

\begin{theorem}\label{Poincare}
	Let $Y$ be an integral projective nodal curve of arithmetic genus $g\geq 3$  defined 
	over an algebraically closed field. Let $n\geq 2$ be an integer and let $L$ be a 
	line bundle of degree $d$ on $Y$. Let $U_L'^s(n,d)$ denote the moduli space of stable vector bundles on $Y$ of rank $n$ and determinant $L$ and let $\mathcal{PU}$ 
	be the projective Poincar\'{e} bundle on $Y\times U_L'^s(n,d)$. Then $\mathcal{PU}_x$ is 
	stable for all $x\in Y_{reg}$ where $Y_{reg}$ is the set of all nonsingular points of $Y$.
\end{theorem}

\begin{proof}
Let $P'$ be a projective subbundle of the restriction of $\mathcal{PU}_x$ to a Zariski open subset $Z$ of 
$U_L'^s(n,d)$ with a complement of codimension at least 2. Then as in the definition of stability for 
projective bundles, we have the following exact sequence 
 
 $
(4.1) \hspace{3.2 cm} 0 \rightarrow \mathcal{T}_{P'/Z}\rightarrow\mathcal{T}_{\mathcal{PU}_x|_{Z}/Z} \rightarrow N_{P'/\mathcal{PU}_x|_Z}\rightarrow 0
$

and we take $N=q_*(N_{P'/\mathcal{PU}_x|_Z})$. Our aim is to conclude $\deg N>0$.
	
	Claim 1: $p^{-1}(Z)$ is a Zariski open subset of $H_x$, and its complement $S$ has codimension at least 2. 
	
	Proof of Claim 1: From  Lemma \ref{L3.5BBN}, we have $p(H_x)$ is non-empty and open in $U_L'^s(n,d)$ which implies $\dim p(H_x)=\dim U_L'^s(n,d)$. Since $Z$ is an open subset of $U_L'^s(n,d)$ with its complement $Z^c$ of codimension $\ge 2$, $p(H_x) \cap Z$ is a non-empty open 
subset of $p(H_x)$ with  {\rm codim}$_{p(H_x)} (Z^c\cap p(H_x)) \geq 2$. Let $p_P: \mathcal{PU}_x \to U_L'^s(n,d)$ be the projection. Then 
$p_P^{-1}(Z^c) = p_P^{-1}(Z)^c$ has codimension $\ge 2$ in $\mathcal{PU}_x$ and hence its intersection $S := p^{-1}(Z)^c$ with the 
open set $H_x$ has codimension $\ge 2$ in $H_x$.   

	Claim 2: For a general $F$, $\psi_{F,x}^{-1}(Z)$ has complement of codimension $\ge 2$ in $P(F_x^*)$.
	
	Proof of Claim 2: Note first that,  since codim $S\geq 2$,  we get $\dim S\leq\dim H_x-2=\dim U_L'^s(n,d)+n-1-2=\dim U_L'^s(n,d)+n-3$.
	The image of $q$ is the subset of $U_{L(x)}'(n,d+1)$ consisting of $(0,1)$-stable vector bundles. Using Lemma ~\ref{01Open}, 
	we know that $q(H_x)$ is a non-empty open subset and hence has  $\dim q(H_x)=\dim U_{L(x)}'(n,d+1)=\dim U_L'^s(n,d)$.
	For a general $F$, we have $\dim(S \cap q^{-1}(F))=\dim q^{-1}(F)+\dim S-\dim H_x \leq \dim q^{-1}(F)-2= n-1-2=n-3$ 
	since the intersection is proper. 
	
             The fibre of $q$ over $F$ is $P(F_x^*)$. The restriction of $p$ to this fibre is $\psi_{F,x}$. Thus $q$ maps the fibre over $F$ 
  isomorphically onto $\psi_{F,x} (P(F_x^*))$. For a general $F$, we have $S \cap q^{-1}F = S \cap P(F_x^*) 
  = P(F_x^*) - \psi_{F,x}^{-1}(Z)$ has dimension  $\le n-3$, i.e., codimension $\ge 2$ in $P(F_x^*)$. 
  	
The complement of $P'$ in $\mathcal{PU}_x|_{Z}$ is open in $\mathcal{PU}_x$ as $\mathcal{PU}_x|_{Z}=p^{-1}(Z)$ and $Z$ is open. 
Since $H_x$ is open in $\mathcal{PU}_x$ and $\mathcal{PU}_x$ is irreducible, $H_x\cap (P')^c \neq \emptyset$. 
For $(E,\ell) \in H_x\cap (P')^c$, $\ell \subset E_x$ is not in the fibre of $P'$ at $x$. 
	
	\noindent It follows from Lemma ~\ref{uniqlem} that $(id_Y\times\psi_{F,x})^*(\mathcal{PU})\cong P(\mathcal{E})$, 
where $\mathcal{E}$ is as defined in \eqref{SESfamily}.\\
	As remarked in \cite[Remark 2.2]{BBN}, there exists a vector subbundle $V'$ of $\mathcal{E}'_x:= 
\mathcal{E}_x\vert_{\psi_{F,x}^{-1}(Z)}$ such that $\psi_{F,x}^*(P')\cong P(V')$. It follows that, if $N$ is defined as 
in Definition \ref{defn}, we have $\psi_{F,x}^*N\cong V'^*\otimes(\mathcal{E}'_x/V')$. 
	
	Claim 3: The condition on $\ell$ gives us that the line subbundle 
$\mathcal{O}_{\mathbf{P}}(1)\vert_{\psi_{F,x}^{-1}(Z)}$ of $\mathcal{E}_x'$ is such that 
$\mathcal{O}_{\mathbf{P}}(1)\vert_{\psi_{F,x}^{-1}(Z)}\nsubseteq V'$.
	
	Proof of Claim 3:  Let $\phi_E$ be the inverse image of $E$ under $\psi_{F,x}$. i.e., we have
	\begin{center}
		\begin{eqnarray}
		\begin{tikzcd}
		&  0\arrow{d} &  0\arrow{d} & & \\
		& F(-x)\arrow{d}\arrow[equal]{r}& F(-x)\arrow{d}& & \\
		0\arrow{r}& E=\ker\phi_E\arrow{r}\arrow{d} & F\arrow{r}{\phi_E}\arrow{d} & k_x\arrow[equal]{d}\arrow{r} & 0\\
		0\arrow{r} & E_x/\ell\arrow{r}\arrow{d}& F_x\arrow{r}{\phi_E}\arrow{d} & k_x\arrow{r}& 0\\
		& 0 &  0 & & 
		\end{tikzcd}
		\end{eqnarray}
	\end{center}
	
	It is enough to show that $\ell$ is the fibre of $\mathcal{O}_{{P}(F_x^*)}(1)$ at $\phi_E$. 
	
	$(\mathcal{O}_{{P}(F_x^*)}(1))_{\phi_E}=V_{\phi_E}^*$ where $V_{\phi}$ denote the 1-dimensional subspace generated 
by $\phi$ in $F_x^*$. We know that, for $v\in F_x$, $v^{**}(\phi_E)=\phi_E(v)$. Let $v^{**}\neq0$ be in the fibre of 
$(\mathcal{O}_{{P}(F_x^*)}(1))$ at $\phi_E$. This implies $\phi_E(v)\neq 0$. i.e, $v\notin Im(E_x/\ell\rightarrow F_x)$ 
and hence $v$ corresponds to the line $\ell$.
		
	\noindent It follows from Lemma ~\ref{subbundle} that 
	\begin{equation*}
	\frac{\deg V'}{\text{rk}\hspace{2pt} V'}<\frac{\deg\mathcal{E}_x}{\text{rk}\hspace{2pt} \mathcal{E}_x},
	\end{equation*}
	and we get $\deg \psi_{F,x}^*N>0$. Let $\bar{N}=i_*(N)$ where $i: U'_L(n,d)\hookrightarrow U_L(n,d)$ is the inclusion, then
$\deg N =\deg\bar{N}$ so that 
$$\deg \psi_{F,x}^*\bar{N} > 0\, .$$
 Since $\theta_L$ is an ample line bundle, by Lemma \ref{psi}, $\psi_{F,x}^*\theta_L$ is an ample line bundle on  
$P(F_x^*)$.  As Pic ${P}(F_x^*) \cong \mathbb{Z}$ and it is generated by $\mathcal{O}_{P(F_x^*)}(1)$, it follows that 
$\psi_{F,x}^*\theta_L$ is isomorphic to $\mathcal{O}_{P(F_x^*)}(r)$ for some $r>0, r \in\mathbb{Z}$. 

        We have $c_1(\bar{N}) = \lambda_N \theta_L, \lambda_N \in \mathbb{Z}.$ Let $u$ denote the dimension of $U_L(n,d)$. Then  
$$\deg\bar{N} = [c_1(\bar{N}).\theta_L^{u-1}](U_L(n,d)) = \lambda_N \theta_L^u(U_L(n,d))= \lambda_N C_U, $$
$C_U$ being a positive constant. We have 
$c_1(\psi_{F,x}^*\bar{N})= \lambda_{\psi_{F,x}^* \bar{N}} \mathcal{O}_{\mathbf{P}(F_x^*)}(1)$ so that 
$$\deg(\psi_{F,x}^*\bar{N}) = \lambda_{\psi_{F,x}^* \bar{N}} > 0\, .$$  
Since $\psi_{F,x}$ has degree $1$, one has $c_1(\psi_{F,x}^*\bar{N})= \psi_{F,x}^*(c_1(\bar{N}))$. 
Now $$\psi_{F,x}^*(c_1(\bar{N})) = \psi_{F,x}^*(\lambda_N \theta_L)= \lambda_N \psi_{F,x}^*(\theta_L) 
= \lambda_N r \mathcal{O}_{P(F_x^*)}(1)$$ 
so that $(\deg \psi_{F,x})^*(c_1(\bar{N})) = r \lambda_N.$ Hence $r \lambda_N =  \lambda_{\psi_{F,x}^* \bar{N}} >0$. 
Therefore $\deg\bar{N} > 0$.
\end{proof}

\begin{theorem}
	Let $Y$ be an integral projective nodal curve of arithmetic genus $g\geq 3$ defined 
	over an algebraically closed field. Let $n\geq 2$ be an integer and let $L$ be a 
	line bundle of degree $d$ on $Y$. Let $U_L'^s(n,d)$ denote the moduli space of stable vector bundles on $Y$ of rank $n$ and determinant $L$ and let $\mathcal{PU}$ 
	be the projective Poincar\'{e} bundle on $Y\times U_L'^s(n,d)$. Let $\eta$ and $\theta_L$ be divisors 
	defining the polarisation on $Y$ and $U_L'^s(n,d)$ respectively. Then $\mathcal{PU}$ is stable 
	with respect to $a\eta+b\theta_L$, $a,b>0$.
\end{theorem}

\begin{proof}
	Suppose $Z\subset Y\times U_L'^s(n,d)$ is a Zariski open subset with a complement of codimension at least 
	$2$ and $P'$ a projective subbundle of the restriction of $\mathcal{PU}$ to $Z$. 
	
	Claim: For general $E\in U_L'^s(n,d)$, $Y\times\{E\}\subset Z$.
	
	The codimension of $Z^c$ in $Y\times U_L'^s(n,d)$ is at least $2$, implies that
	$$\dim(Z^c)\leq \dim(U_L'^s(n,d))-1.\hspace{3 cm} (\ast)$$
	
	Consider the  projection map $p: Y\times U_L'^s(n,d) \to  U_L'^s(n,d)$. 
	
	For all  $E\in p(Z^c)^c \subset U_L'^s(n,d)$, we have  $Y\times \{E\}\subset Z$.
	To show  this holds for general $ E\in U_L'^s(n,d)$, it is enough to show that $\dim p(Z^c)< \dim U_L'^s(n,d)$ which 
	follows from $\ast$.
	
	It is also easy to see that, for general $x\in Y$, the codimension of the complement of $\{x\}\times (U_L'^s(n,d)\cap Z)$ 
	in $\{x\}\times U_L'^s(n,d)$ is at least $2$.

	Let $N$ be as in Definition \ref{defn}. By applying Theorem \ref{Poincare},  we  have $\deg N_x>0$ for $x\in Y_{reg}$. 
	Also stability of $E$ implies that $\deg N|_{Y\times\{E\}}>0$.
	
	Let $\bar{N}=(Id\times i)_*N$. So by definition of degree we get, 
		$\deg N=\deg\bar{N}$. Since the codimension of the complement of $\{x\}\times (U_L'^s(n,d)\cap Z)$ 
	in $\{x\}\times U_L'^s(n,d)$ is at least $2$,  $\deg N_x= \deg\bar{N}_x $ for a general $x$.
	
	Now, if $u$ is the dimension of $U_L(n,d)$ then 
	\begin{align*}
	 \deg N=\deg \bar{N} &=[c_1(\bar{N})\cdot(a\alpha+b\theta_L)^u](Y\times U_L(n,d)) & \\
	&= [c_1(\bar{N})\cdot(\lambda\alpha.\theta_L^{u-1}+\mu\theta_L^u)](Y\times U_L(n,d)) & \text{for some }\lambda,\mu>0 \\
	&=\lambda\alpha(Y)\cdot\deg\bar{N}_x+\mu\deg (N|_{Y\times\{E\}})\cdot\theta_L^u(U_L(n,d)) & \\
	&>0
	\end{align*}	
\end{proof}

\section{Projective Picard bundle}

The construction of the projective Picard bundle follows exactly as in \cite{BBN}.
We outline it here for the sake of completeness.

           Recall that  $U_Y'^s(n,d)$ is a GIT quotient of $R'^s$ by $PGL(M)$, let $\pi: R'^s\rightarrow U_Y'^s(n,d)$ be the quotient map.  
  There is a (locally universal) vector bundle  $\mathcal{E}_{R'^s} \to Y \times R'^s$. 
  Let $p_{R'^s}: Y \times R'^s \rightarrow R'^s $ denote the projection to the second factor. Then ${p_{R'^s}}_{*}\mathcal{E}_{R'^s}$ is a torsion-free sheaf on $R'^s$ and is non-zero if and only if $d > n(g-1)$ \cite[Theorem 1.3]{UB1}. For $d \geq 2n(g-1), n \ge 2$, the set 
\begin{equation*}
\{r\in R'^s \vert ~H^1(Y,\mathcal{E}_{R'^s}|_{Y\times\{r\}})=0\}
\end{equation*}
equals $R'^s$ and ${p_{R'^s}}_{*}\mathcal{E}_{R'^s}$ is a vector bundle on $R'^s$.  
The projective Picard bundle on $U_Y'^s(n,d)$ is defined to be the quotient of $P({p_{R'^s}}_*(\mathcal{E}_{R'^s}))$ 
by the action of $PGL(M)$ on it and is denoted by $\mathcal{PW}$.

\begin{lemma}\label{uniqlem2}
	Let $\mathcal{E}$ be a vector bundle on $Y\times Z$ such that the restriction of $\mathcal{E}$ to $Y\times\{z\}$ is stable of rank $n$ and determinant $L$ for all $z\in Z$, and let $\psi_{\mathcal{E}}: Z\rightarrow U_Y'^s(n,d)$ be the corresponding morphism. Also suppose that $H^1(Y,\mathcal{E}|_{Y\times\{z\}})=0$ for all $z\in Z$. Then the projective bundles $P({p_Z}_*\mathcal{E})$ and $\psi_{\mathcal{E}}^*(\mathcal{PW})$ are isomorphic.
\end{lemma}

\begin{proof}
	See \cite[Proposition 4.2]{BBN} for a proof.
\end{proof}

An analogue of Corollary \ref{uniqcor} shows that $\mathcal{PW}$ is independent of the choice of $R'$ and $\pi$. Its restriction $\mathcal{PW}|_{U_L'^s(n,d)}$ will be called the projective Picard bundle on $U_L'^s(n,d)$.

\subsection{Stability of the projective Picard bundle}

\begin{theorem} \label{Picard}
	Let $Y$ be an integral projective nodal curve of arithmetic genus $g\geq 3$ 
	defined over an algebraically closed field. Let $n\geq2$ be 
	an integer and let $L$ be a line bundle of degree $d$ on $Y$. 
	Let $U_L'^s(n,d)$ denote the moduli space of stable vector bundles 
	on $Y$ of rank $n$ and determinant $L$, suppose further that $d\ge 2n(g-1)$. 
	Then the projective Picard bundle $\mathcal{PW}|_{U_L'^s(n,d)}$ is stable.
\end{theorem}

\begin{proof}
The proof is exactly on the same lines as in the case when $Y$ is smooth  \cite{BBN}.  
Let $P'$ be a projective subbundle of the restriction of $\mathcal{PW}$ to 
some Zariski-open subset $Z'$ of $U_L'^s(n,d)$ with a complement of codimension 
at least 2. For a nonsingular point $x\in Y$ and a $(0,1)$-stable vector bundle $F$ 
of rank $n$ and determinant $L(x)$, we have a morphism, 
$\psi_{F,x}:{P}(F_x^*)\rightarrow U_L'^s(n,d)$ (see Section \ref{psiFx}).  
For a general $F$, the complement of $Z := \psi_{F,x}^{-1}(Z')$ has codimension at least 
$2$. By \eqref{El}, $F$ corresponds to a pair $(E,\ell)$ with $E \in U_L'^s(n,d)$ and $\ell \subset E_x$ is a line.    

          We  restrict  Diagram \ref{Diagram1} to $Y\times Z$. We want to take the direct image by the projection $p_2: Y\times Z\rightarrow Z$. We first note the following observations.

\begin{itemize}
	\item ${p_2}_*(p_1^*F(-x))_z= H^0(p_1^*F(-x)|_{Y\times\{z\}})= H^0(Y,F(-x))$ for all $z\in Z$,
	
	\item Similarly  ${p_2}_*(p_1^*F)_z = H^0(Y,F)$ for all $z\in Z.$
	
	\item One has $i_* \mathcal{O}_{\mathbf{P}}(1)\vert_Z = (i_Z)_*\mathcal{O}_Z(1)$, $i_Z$ 
	being the inclusion map restricted to $Z$. Since $p_2 \circ i_Z = id_Z$,  
	 ${p_2}_*((i_Z)_*\mathcal{O}_{{Z}}(1))= \mathcal{O}_{Z}(1)$. 
	\item Similarly ${p_2}_ * {i}_*(\Omega^1_{P}(1)\vert_{Y \times Z}) = \Omega^1_{Z}(1)$
	
\end{itemize}

Taking the direct image by the projection $p_2$ and using these observations, we have the following diagram 
of exact sequences on $Z$:     

\begin{center}
	\begin{eqnarray}\label{Diagram2}
	\begin{tikzcd}
	&  0\arrow{d} &  0\arrow{d} & & \\
	& H^0(Y,F(-x))\otimes\mathcal{O}_{Z}\arrow{d}\arrow[equal]{r}& H^0(Y,F(-x))\otimes\mathcal{O}_{Z}\arrow{d}& & \\
	0\arrow{r}& {p_2}_*(\mathcal{E}|_{Y\times Z})\arrow{r}\arrow{d} & H^0(Y,F)\otimes\mathcal{O}_{Z}\arrow{r}\arrow{d} & \mathcal{O}_{Z}(1)\arrow[equal]{d}\arrow{r} & 0\\
	0\arrow{r} & \Omega^1_{Z}(1)\arrow{r}& F_x\otimes_{\mathbb{C}}\mathcal{O}_{Z}\arrow{r} & \mathcal{O}_{Z}(1)\arrow{r}& 0
	\end{tikzcd}
	\end{eqnarray}
\end{center}

Now $\mathcal{E}|_{Y\times Z}$ satisfies the hypothesis of Lemma \ref{uniqlem2}, and hence, we have $P({p_2}_*(\mathcal{E}|_{Y\times Z}))\cong (\psi_{F,x}^*\mathcal{PW})|_{Z}$. Now $(\psi_{F,x}^*P')|_{Z}$ being a subbundle of $P({p_2}_*(\mathcal{E}|_{Y\times Z}))$ can be realised as $P(V)$ for some vector subbundle $V$ of ${p_2}_*(\mathcal{E}|_{Y\times Z})$. 

                 We claim that the image of $V$ in $\Omega^1_{Z}(1)$ is non-zero; we prove the claim. Any $0 \neq v \in V_E$ corresponds to a section $s \in H^0(Y, E)$. As in the beginning of \cite[Subsection (4.1)]{UB1}, one can find a general line $\ell \subset E_x$ such that $s(x) \notin \ell$.  Note that $\ell$ is the image of $F(-x)_x$ in $E_x$ in Diagram \ref{CD} and since $s(x)\notin \ell$,  $s\notin H^0(Y,F(-x)) \subset H^0(Y,E)$. The claim follows.
                
                 From the last row of diagram \ref{Diagram2}, one sees that  $\deg(\Omega_{Z}^1(1))$ is $-1$. Since $\Omega_{Z}^1$ is a stable bundle,  we get that the degree of $\text{Im}(V \rightarrow \Omega_{Z}^1(1)) \leq -1$, since the 
  $\ker(V \rightarrow \Omega_{Z}^1(1))$ is a subsheaf of the trivial sheaf, its degree is at most 0 and hence we have 
 $\deg V \leq -1\, .$ i.e.
 $$\deg V^* \ge 1\, .$$
               Also, from the second row of Diagram \ref{Diagram2}, we get $\deg {p_2}_*(\mathcal{E}|_{Y\times Z})= -1\, .$ 
Hence 
$$\deg({p_2}_*(\mathcal{E}|_{Y\times Z})/V) \ge 0\, .$$ 
  
      Now as in the beginning of the proof of Theorem \ref{Poincare}, we have $\psi_{F,x}^*N=V^*\otimes {p_2}_*(\mathcal{E}|_{Y\times Z})/V$.
\begin{equation*}
\deg \psi_{F,x}^*N = \deg(V^*) \text{rk}({p_2}_*(\mathcal{E}|_{Y\times Z})/V) + \deg({p_2}_*(\mathcal{E}|_{Y\times Z})/V) \text{rk} V^* \ge \text{rk}({p_2}_*(\mathcal{E}|_{Y\times Z})/V) > 0
\end{equation*}

Similar arguments as at the end of  the proof of Theorem \ref{Poincare} show that $\deg \psi_{F,x}^*N >0$ implies that $\deg N >0$. This proves the stability of $\mathcal{PW}\vert_{U_L'^s(n,d)}$.
\end{proof}

\begin{remark}
        In case $(n,d)= 1$, there is a Poincar\'e bundle $\mathcal{U}$ on $Y \times U'_L(n,d)$ (respectively a Picard bundle $\mathcal{W}$ on $U'_L(n,d)$) with  the associated projective Poincar\'e bundle $\mathcal{PU}$ (respectively the projective Picard bundle $\mathcal{PW}$).  In view of Remark \ref{rmk1}(1), the stability of the projective Poincar\'e bundles $\mathcal{PU}, \mathcal{PU}_x$ is equivalent to the stability of  the  Poincar\'e bundles $\mathcal{U}, \mathcal{U}_x$ and the stability of the projective Picard bundle $\mathcal{PW}$ is equivalent to the stability of the Picard bundle $\mathcal{W}$. 
\end{remark}

\section{Codimension of the stable moduli space}  \label{codimstable}

         In this section, we assume that $n$ and $d$ are non-coprime. Recall that $U'_Y(n,d)$ denotes the moduli space of vector bundles of rank $n$, degree $d$ on the curve $Y$ and $U^{'s}_Y(n,d)$ its open subset consisting of stable vector bundles. 
For a fixed line bundle $L$ of degree $d$ on $Y$, $U'_L(n,d)$ denotes the moduli space of vector bundles of rank $n$ with determinant $L$ and $U^{'s}_L(n,d)$ its open subset corresponding to stable vector bundles. We estimate the codimension (in $U'_L(n,d)$) of the complement $U'_L(n,d) - U^{'s}_L(n,d)$. 
 
\begin{theorem} \label{thmcodims}
	Let $Y$ be an integral nodal curve of arithmetic genus $g$. 
\begin{enumerate}
\item For $g \ge 2$ and  $n \ge 3$ (resp. $n = 2$),
	$${\rm codim}_{U'_Y(n,d)}  (U'_Y(n,d)- U^{'s}_Y(n,d)) \ge 2(g-1) (\ {\rm resp.} \  \ge g-1)\, .$$
\item For $g \ge 2$ and  $n \ge 3$ (resp. $n = 2$),
	$${\rm codim}_{U'_L(n,d)}  (U'_L(n,d)- U^{'s}_L(n,d)) \ge 2(g-1) (\ {\rm resp.} \  \ge g-1)\, .$$
\end{enumerate}
\end{theorem}

\begin{proof}
(1)	A semistable vector bundle that is not stable has a filtration of length $r$ of the form 
	\begin{equation}\label{Fr}
	0=E_0\subset E_1\subset\dots\subset E_r=E
	\end{equation}
	where 
	\begin{enumerate}
		\item $E_i$ are semistable torsion-free sheaves with $\mu(E_i)=\mu(E)$ for all $i=1,\dots,r$.
		\item $F_i=E_i/E_{i-1}\in U_Y(n_i,d_i)$ are stable torsion-free sheaves.
	\end{enumerate}
Note that $\mu(E_i)=\mu(E)$ for all $i$, implies $\mu(F_i)=\mu(E)$ for all $i=1,2,\dots,r$. 
Hence $n_id_j -n_j d_i =0$ for all $1 \le i, j \le r$. 
 
                     For $r \ge 2$, let $S_r$ denote the subset of $U'_Y(n,d)$ consisting of vector bundles $E = E_r$ having a filtration \eqref{Fr} with $F_i \in U_Y(n_i,d_i)$ and let $e_r$ denote the dimension of the subset $S_r$ in $U'_Y(n,d)$.

We first consider the case when the length of the filtration $r$ is 2.
\begin{equation}\label{F2}
(\mathcal{F}_2): \ \ 0=E_0\subset E_1\subset E_2 = E
\end{equation}
such that $F_1=E_1$ and $F_2=E/E_1$.

We have the following exact sequence
\begin{equation}\label{SESF1F2}
0\rightarrow F_1\rightarrow E\rightarrow F_2\rightarrow 0
\end{equation}
and to estimate the number of parameters determining such extensions, we compute $\dim Ext^1(F_2,F_1)$.

Note that $\deg(Hom(F_2,F_1))= n_2d_1-n_1d_2+\sum_kb_k(F_1)b_k(F_2)$ and  $\mu(F_1)=\mu(F_2)$ implies 
\begin{equation*}
\deg(Hom(F_2,F_1))=\sum_kb_k(F_1)b_k(F_2). 
\end{equation*}

Therefore,

\begin{align*}
\dim Ext^1(F_2,F_1) &= h^1(Hom(F_2,F_1)) +2\sum_kb_k(F_1)b_k(F_2)\\
 &= h^0(Hom(F_2,F_1))-\chi(Hom(F_2,F_1)) +2\sum_kb_k(F_1)b_k(F_2)\\
 &= h^0(Hom(F_2,F_1))-[\deg(Hom(F_2,F_1))+n_1n_2(1-g)] +2\sum_kb_k(F_1)b_k(F_2)\\
 &=n_1n_2(g-1)+\sum_kb_k(F_1)b_k(F_2) + h^0(Hom(F_2,F_1))
\end{align*}

Since $F_1$ and $F_2$ are stable sheaves with same slope, $h^0(Hom(F_2,F_1))= 1$(resp. 0) when $F_2\cong F_1$ (resp. $F_2\ncong F_1$).

Recall that $S_2$ denotes the subset of $U'_Y(n,d)$ consisting of vector bundles $E$ having a filtration \eqref{F2} with $F_i \in U_Y(n_i,d_i)$ and $e_2$ denotes the dimension of the subset $S_2$. We have
\begin{align*}
e_2 &\le n_1^2(g-1)+1-\sum_k b_k(F_1)^2 + n_2^2(g-1)+1-\sum_kb_k(F_2)^2 + \dim \mathbb{P}(Ext^1(F_2,F_1))\\
 & = (n_1^2+n_2^2)(g-1)+2-\sum_k b_k(F_1)^2 -\sum_kb_k(F_2)^2 + n_1n_2(g-1) \\
 & + \sum_kb_k(F_1)b_k(F_2) + h^0(Hom(F_2,F_1)) -1\\
 & =(n_1^2+n_2^2)(g-1)+2+n_1n_2(g-1)-\sum_k[b_k(F_1)-b_k(F_2)]^2-\sum_kb_k(F_1).b_k(F_2)\\
 & + h^0(Hom(F_2,F_1))-1
\end{align*}
Here $k$ varies over the nodes. If $F_2\ncong F_1$, $h^0(Hom(F_2,F_1))$=0 and
\begin{equation}
e_2 \le (n_1^2+n_2^2)(g-1)+2+n_1n_2(g-1)-\sum_k[b_k(F_1)-b_k(F_2)]^2-\sum_kb_k(F_1).b_k(F_2)-1
\end{equation}

Since $b_k(F_i) \ge 0$ and the terms with $b_k(F_i)$ appear with a negative sign, we note that $e_2$ is maximal if $b_k(F_i)=0$. i.e., $F_i$s are vector bundles. Hence we have

\begin{equation} \label{e2}
e_2\leq (n_1^2+n_2^2)(g-1)+1+n_1n_2(g-1)
\end{equation}

Let $S_2^{12}$ denote the subset of $U_Y'(n,d)$ corresponding to vector bundles given by the short exact sequence  \eqref{SESF1F2} with $F_1\cong F_2$. Then the pair $(F_1,F_2)$ belongs to the diagonal $\Delta$ of $U_Y(n_1,d_1)\times U_Y(n_2,d_2)$. Let $e_2^{12}$ denote the dimension of $S_2^{12}$. Then 

\begin{align*}
e_2^{12}&= \dim \Delta+ \dim \mathbb{P}(Ext^1(F_2,F_1))\\
 &=n_1^2(g-1)+1+n_1n_2(g-1)+1-1
\end{align*}

Note that $e_2-e_2^{12} = n_2^2(g-1)\geq 1$ for $g\geq2$.

Hence the codimension in $U_Y'(n,d)$ of the subset corresponding to vector bundles $E$ given by \eqref{SESF1F2} is
\begin{align*}
 &\geq n^2(g-1)+1-e_2\\
 & =n_1n_2(g-1)
\end{align*}

In particular, for $n=2$, we have $n_1=n_2=1$ and
\begin{equation}\label{rank2}
{\rm codim}_{U'_Y(n,d)} (U'_Y(n,d)- U^{'s}_Y(n,d)) \ge g-1
\end{equation}
\begin{equation}\label{r2n3}
{\rm For} \  n\ge 3, \  
{\rm {\rm codim}} \ S_2 \ge 2(g-1)\, .
\end{equation}

                   We can now assume that $r \ge 3, n \ge 3$. A vector bundle $E = E_r$ having a filtration \eqref{Fr}   comes in the following exact sequence
\begin{equation}\label{SESEr}
0 \rightarrow E_{r-1}\rightarrow E_r\rightarrow F_r\rightarrow 0
\end{equation}
with $E_{r-1}$ varying on the space $S_{r-1}$. 

        As seen in case $r=2$, to estimate the dimension of space of extensions, we may assume that $E_i$s are locally free.  The group $G_r = {\rm Aut}\ E_{r-1} \times  {\rm Aut}\ F_r$ acts on these extensions, and $E_r$ given by the extensions in an orbit of $G_r$ are all isomorphic. One has $\dim ({\rm Aut}\ F_r) = 1$. Since ${\rm End}\ E_{r-1} \supset \mathbb{C}^*Id \oplus {\rm Hom} (F_{r-1}, E_{r-2})$, it follows that $\dim ({\rm Aut}\ E_{r-1}) \ge 1+ h^0(F_{r-1}^* \otimes E_{r-2})$ so that dim $G_r \ge 
2+  h^0(F_{r-1}^* \otimes E_{r-2})$. 
    
Hence 
$e_r \le e_{r-1} + n_r^2(g-1)+1 + h^1(F_{r}^* \otimes E_{r-1}) - \ {\rm dim} \ G_r\, . $ 
We have 
$$h^1(F_{r}^* \otimes E_{r-1}) = - \chi(F_{r}^* \otimes E_{r-1}) + h^0(F_{r}^* \otimes E_{r-1}) 
      = (\sum_{i= 1} ^{r-1} n_i)n_r (g-1) + h^0(F_{r}^* \otimes E_{r-1})\, .$$
 Hence for any $r$ ($r\ge 3$),
\begin{equation}\label{er}
e_r \le e_{r-1} + n_r^2(g-1) -1 +  (\sum_{i= 1} ^{r-1} n_i)n_r (g-1)  
                 + h^0(F_{r}^* \otimes E_{r-1})  - h^0(F_{r-1}^* \otimes E_{r-2}).
\end{equation} 
            We replace $r$ by $r-1$ in the inequality \eqref{er} and substitute it for $e_{r-1}$ in the inequality \eqref{er} to get                         
$$
\begin{array}{ccll}

 e_r  &  \le & e_{r-1} + n_r^2(g-1) -1+   (\sum_{i= 1} ^{r-1} n_i)n_r (g-1)  
                 &+ h^0(F_{r}^* \otimes E_{r-1})  - h^0(F_{r-1}^* \otimes E_{r-2}) \\     
   {}   &   \le  & n_r^2(g-1) -1+ (\sum_{i= 1} ^{r-1} n_i)n_r (g-1)  
                 & + h^0(F_{r}^* \otimes E_{r-1})  - h^0(F_{r-1}^* \otimes E_{r-2})\\
   {}  &  {} &+e_{r-2} + n_{r-1}^2(g-1) -1+ (\sum_{i= 1} ^{r-2} n_i)n_{r-1} (g-1)  
                & + h^0(F_{r-1}^* \otimes E_{r-2})  - h^0(F_{r-2}^* \otimes E_{r-3}).\\
\end{array}
$$
We now substitute for $e_{r-2}$ (using inequality \eqref{er})  in the resulting inequality and go on like this till we get $e_2$, then we substitute for $e_2$ from equation \eqref{e2}. 
We finally get
$$
\begin{array}{ccll}
  e_r  &   \le  & n_r^2(g-1) -1  &+ (\sum_{i= 1} ^{r-1} n_i)n_r (g-1)  \\
  {}  &  {}     & + n_{r-1}^2(g-1) -1  &+ (\sum_{i= 1} ^{r-2} n_i)n_{r-1} (g-1)  \\
  {}   &  {}  &+ \cdots - \cdots   &+ \cdots   \\
  {}  &  {}  & + n_4^2(g-1) -1 &+ (\sum_{i= 1} ^{3} n_i)n_4 (g-1)  \\
  {}  &  {}  & + n_3^2(g-1) -1 &+ (\sum_{i= 1} ^{2} n_i)n_3 (g-1)  \\
  {}  &  {}  & +(n_1^2+n_2^2)(g-1)  &+ n_1n_2(g-1) + 1 + \delta_r\\
  {}  &  =  &  \sum_{i=1}^r n_i^2 (g-1) - (r-3) &+ \sum_{1\le i <j \le r} n_i n_j (g-1) +\delta_r
\end{array}         
$$
where 
$$
\begin{array}{ccll}
\delta_r & =  & h^0(F_{r}^* \otimes E_{r-1})  & - h^0(F_{r-1}^* \otimes E_{r-2})\\
     {}  & {}    & + h^0(F_{r-1}^* \otimes E_{r-2})  & - h^0(F_{r-2}^* \otimes E_{r-3})\\
{} & {}         &  + \cdots & - \cdots \\
{} & {}        &+ h^0(F_{4}^* \otimes E_{3})  & - h^0(F_{3}^* \otimes E_{2})\\ 
{} & {}         &+ h^0(F_{3}^* \otimes E_{2}) &  - h^0(F_{2}^* \otimes E_{1})\\ 
{} & =        &  h^0(F_{r}^* \otimes E_{r-1}) & - h^0(F_{2}^* \otimes E_{1})\\
{}  & \le  & r-1. & {}
\end{array} 
$$
Note that in the expression for $\delta_r$, the $i$th term in the second column cancels with 
the $(i+1)$th term in the first column. We have also used the fact that 
$$h^0(F_{r}^* \otimes E_{r-1}) \le \sum_{i=1}^{r-1} h^0(F_{r}^* \otimes F_{i}) \le r-1\, .$$ 
Hence 
$$
\begin{array}{ccll}
 e_r  &  \le  &  \sum_{i=1}^r n_i^2 (g-1) - (r-3) + \sum_{1\le i <j \le r} n_i n_j (g-1) & +(r-1)\\
  {}  &  =  &  \sum_{i=1}^r n_i^2 (g-1) + \sum_{1\le i <j \le r} n_i n_j (g-1)  & +2. \\
\end{array}
$$                 
Therefore, 
$$
\begin{array}{ccl}
{\rm codim} \ S_r & \ge & (\sum_{i=1}^r n_i)^2 (g-1) +1 -  \sum_{i=1}^r n_i^2 (g-1) - \sum_{1\le i <j \le r} n_i n_j (g-1) - 2\\
{}   & =  & \sum_{1\le i < j \le r} n_i n_j (g-1) -1\\
{}   &  \ge  & 3(g-1) -1\ {\rm for} \  r \ge 3.
\end{array}
$$
Thus 
$${\rm codim} \ S_r \ge 3(g-1)-1 \ge 2(g-1) \ {\rm for} \  r \ge 3, g \ge 2\, .$$
 This, together with equations \eqref{rank2} and \eqref{r2n3}, proves the first part of the theorem.  \\
 (2) There is a determinant homomorphism $det: U'_Y(n,d)  \to J(Y)$ defined by $E \mapsto \wedge^n E$ with all its fibres
isomorphic. Hence Part 2 follows from Part 1. 
\end{proof}

\begin{corollary} \label{coro1}
Let $Y$ be an integral nodal curve of arithmetic genus $g\geq 2$. 
Assume that if $n=2$ and $g =2$ then $d$ is odd. Then 
\begin{enumerate}
\item $\ {\rm Pic} \  U^{'s}_L(n,d) \cong \mathbb{Z}\, .$
\item $\ {\rm Pic} \  U'_L(n,d) \cong \mathbb{Z}\, .$ 
\item $\ {\rm The \ class \ group} \ Cl(U_L(n,d)) \cong \mathbb{Z}\, .$
 ~~~ $ \ {\rm The \ class \ group} \ Cl(U'_L(n,d)) \cong \mathbb{Z}\, .$
 \end{enumerate}
\end{corollary}

\begin{proof}
 (1)        Let $\widetilde{U}_L(n,d)$ denote a normalisation of $U_L(n,d)$ and $\pi: \widetilde{U}_L(n,d) \rightarrow U_L(n,d)$ the normalisation map.  Since $\pi$ is a finite map, ${\rm codim}_{\widetilde{U}_L(n,d)} (\widetilde{U}_L(n,d)- \pi^{-1}U^{'s}_L(n,d)) = {\rm codim}_{U_L(n,d)} (U_L(n,d)- U^{'s}_L(n,d)) \ge 2$  by Theorems \ref{thmcodim'} and \ref{thmcodims}.  
Since $\widetilde{U}_L(n,d)$ is normal, this implies that the restriction map 
$res_P: \ {\rm Pic} \ \widetilde{U}_L(n,d) \to \ {\rm Pic} \ \pi^{-1}U^{'s}_L(n,d)$ is injective.  Note that $\pi$ is an isomorphism over $U^{'s}_L(n,d)$ so that ${\rm Pic} \ \pi^{-1}U^{'s}_L(n,d)= \ {\rm Pic} \ U^{'s}_L(n,d)$. 
 Hence ${\rm Pic} \ U^{'s}_L(n,d)$ has rank $\ge 1$. By \cite[Proposition 2.3]{UB4}, for $g \ge 2$, 
we have ${\rm Pic} \ U^{'s}_L(n,d) \cong \mathbb{Z} \ {\rm or} \ \mathbb{Z}/q\mathbb{Z}$ for some integer $q$. It follows that  
${\rm Pic} \ U^{'s}_L(n,d) \cong \mathbb{Z}$.  \\
(2)  The proof of Part 1 also implies that ${\rm Pic} \  \widetilde{U}_L(n,d) \cong \mathbb{Z}$. 
The restriction map $res_P$ in the proof of Part (1) factors through the injective restriction map 
$\ {\rm Pic} \ U'_L(n,d) \hookrightarrow \ {\rm Pic} \ U^{'s}_L(n,d) \cong \mathbb{Z}$ (by Theorem \ref{thmcodims}). 
It follows that $\ {\rm Pic} \ U'_L(n,d) \cong \mathbb{Z}\, .$  \\
(3) Since $U_L(n,d)$ is nonsingular in codimension $1$, by Theorems \ref{thmcodim'} and \ref{thmcodims}, 
by \cite[Proposition 6.5, p.133]{Ha}, we have an isomorphism of class groups
       $$Cl(U_L(n,d)) \cong Cl(U^{'s}_L(n,d))\, .$$ Since $U^{'s}_L(n,d)$ is nonsingular, one has 
$Cl(U^{'s}_L(n,d)) \cong \ {\rm Pic} \ U^{'s}_L(n,d) \cong \mathbb{Z}$ by Part (1).
\end{proof}

\section{Non-existence of Poincar\'e bundle in the non-coprime case} \label{nonexist}

            If $(n,d) =1$, one has $U^{'s}_Y(n,d) = U'_Y(n,d)$ and there is a Poincar\'e bundle on $U'_Y(n,d)\times Y$ \cite[Section 7,Theorem 5.2']{N}.  
In this section, we show that if $n$ and $d$ are not coprime, then there is no  Poincar\'e bundle on $V\times Y$ for any Zariski open subset of $V \subset U^{'s}_Y(n,d)$.              
We follow the proof of Ramanan for the non-existence of Poincar\'e bundle in the non-coprime case over smooth curves \cite[Theorem 2]{R}. Hence we retain some of his notations, and the proofs of lemmas whose proofs mostly go through in the nodal case are only sketched. 
There are some additional results and many modifications needed in the nodal case, which we give in greater detail.

Let $\mathcal{O}_Y(1)$ denote an ample line bundle on $Y$; we can assume that its degree is $1$. 
For a vector bundle $E$, we write $E(m) = E \otimes  \mathcal{O}_Y(m)$.

We start with some results on the codimensions needed in the nodal case.

\subsection{ Codimensions of subsets of moduli spaces} \label{codimensions}
 
The moduli space $U_X = U_X(n, d)$ of semistable vector bundles of rank $n$ and degree $d$ on $X$ is the GIT quotient of an open subset of the
quot scheme Quot of coherent quotients ${\mathcal O}_X^N \rightarrow E \rightarrow 0$ with fixed Hilbert polynomial 
$P(m) = mn + d + n(1 -g), n = P(0)$, assume $d$ sufficiently large.     
 Let $R$ be the open smooth subscheme of Quot corresponding to vector bundles $E$ with $H^1(E) = 0$
and $H^0(X,E)\otimes {\mathcal O}_X \cong {\mathcal O}_X^N$. Let $R^s$ and $R^{ss}$ denote the subsets of $R$ consisting of
stable and semistable points respectively for the action of $PGL(N)$. 
Let $U^s_X$ be the open subset of $U_X$ corresponding to stable vector bundles.  
 Then $U_X$ (resp. $U_X^s$) is the GIT quotient of
$R^{ss}$ (resp. $R^s$) by $PGL(N)$. 

\begin{proposition} \label{Prop1.2} (\cite[Proposition 1.2.]{UB4}).

   Let $g(X)$ be the genus of $X$. For $g(X) \ge 2$ and  $ n \ge 3$ (resp. $n=2$) one has

\begin{enumerate}
\item  ${\rm codim}_R(R - R^s) \ge 2g(X)-2 (resp. \ge g(X)-1)$ 
\item ${\rm codim}_R(R - R^{ss}) \ge 2g(X)-1 (resp. \ge g(X))$ 
\item ${\rm codim}_{R^{ss}} (R^{ss} - R^s) \ge 2g(X)-2 (resp. \ge g(X)-1)$. 
\end{enumerate}
\end{proposition}
\begin{corollary}\label{Coro1.3} (\cite[Corollary 1.3]{UB4}).
 Fix a line bundle $L$ on $X$. Let $R_L \subset R, R^s_L \subset R^s, R^{ss}_L\subset R^{ss}$
be the closed subvarieties corresponding to bundles $E$ with fixed determinant $L$.  
Then for $n \ge 3$ (resp. $n=2$)  one has: \\
\begin{tabular}{lll}
(1) &  ${\rm codim}_{R_L}(R_L - R^{ss}_L) \ge 2g(X)-1 $ & (resp. $\ge g(X))$ \\[1mm]
(2)  & ${\rm codim}_{R_L}(R_L - R^{s}_L) \ge 2g(X)-2 $ & (resp. $\ge
g(X)-1)$ \\ [1mm]
(3) & ${\rm codim}_{R_L^{ss}}(R_L^{ss} - R^{s}_L) \ge 2g(X)-2 $ & (resp. $\ge
g(X)-1)$.
\end{tabular}
\end{corollary}

    There is a universal bundle ${\mathcal E} \rightarrow R \times X$.  
    For $j = 1, \cdots, J$ let ${\mathcal E}_j = (p_R)_* ({\mathcal E}\vert_{R \times D_j}), D_j = x_j + z_j$, a divisor on
$X$ consisting of two closed points.  Let $Gr_j$ denote the Grassmannian bundle of $n$-dimensional quotients of ${\mathcal E}_j$ and let $\widetilde{R}$ be the fibre product of $Gr_j$ over $R$. 

Let $H$ be the closed subscheme of $\widetilde{R}$ corresponding to Generalised parabolic bundles (GPBs)  $(E, F^j(E))$ with fixed determinant $\overline{L}$ \cite{UB4} .  Then  $\widetilde{R}$ (resp. $H$) is a fibre
bundle over $R$ (resp. $R_L$) with fibres isomorphic to 
$\prod_j Gr (n, 2n)$ (resp. $\prod_j H_j$ where $H_j$ is a smooth  hyperplane section of $Gr(n, 2n))$.   
 Let $H'$ (resp. $\widetilde{R'}$ ) be the open subscheme of $H$  (resp. $\widetilde{R}$) corresponding to GPBs $(E, F^j(E))$ such that the projections $p_j, p'_j$  from $F^j(E)$ to $E_{x_j}$ and $E_{z_j}$ are isomorphisms.  
 Let the superscripts $ss$ and $s$ denote semistable and stable points respectively. 

      Let $\overline{R}^{'s} \subset \widetilde{R'}$ (respectively $\overline{R}^{'ss} \subset \widetilde{R'}$) be the subset corresponding to GPBs $(E, F^j(E))$ where the underlying bundle $E$ is stable (respectively semistable), 
subsets $\overline{H}^{'s}, \overline{H}^{'ss}$ of $H'$ are similarly defined and all these subsets are $PGL(N)$-invariant.    
The moduli spaces $M$ (resp. $M_{\overline{L}}$) of GPBs on $X$ of rank
$n$, degree $d$ (resp. with fixed determinant $\overline{L}$) are GIT 
quotients of $\widetilde{R}^{ss}$  (resp. $H^{ss}$) by $PGL(N)$.
 Let $M', M^{'s}, \overline{M}^{'s}, \overline{M}^{'ss}$ and $M'_{\overline{L}}, M^{'s}_{\overline{L}}, \overline{M}^{'s}_{\overline{L}}, \overline{M}^{'ss}_{\overline{L}}$ 
 be their open subsets which are quotients of $\widetilde{R}^{'ss}, \widetilde{R}^{'s}, \overline{R}^{'s}, \overline{R}^{'ss}$ 
and $H^{'ss}, H^{'s}, \overline{H}^{'s}, \overline{H}^{'ss}$ respectively.  

\begin{proposition} \label{prop1.5}
 For $n \geq 3$  (resp. $n = 2$) one has :\\
\begin{tabular}{lll}
(1) & {\rm codim}$_{\widetilde{R}'}(\widetilde{R}' -
\overline{R}^{'ss}) \ge 2g(X) -1$  & (resp $g(X)$) \\
(2) & {\rm codim}$_{\widetilde{R}'}(\widetilde{R}' -
\overline{R}^{'s}) \ge  2g(X)-2$  & (resp $g(X) -1$) \\ 
(3) & {\rm codim}$_{\overline{R}^{'ss}}(\overline{R}^{'ss} -
\overline{R}^{'s}) \ge 2g(X)-2$  & (resp $g(X) -1$)
\end{tabular}
\end{proposition}

\begin{proof}
(1) Since $pr : \widetilde{R}' \rightarrow R$ is a fibre bundle with
fibres isomorphic to $\prod_j GL(n)$, from Proposition \ref{Prop1.2} it
follows that  
$${\rm codim}_{\widetilde{R}'}(\widetilde{R}' -
\overline{R}^{'ss}) =  \ {\rm codim}_{\widetilde{R}'} (\widetilde{R}' - 
pr^{-1} ({R}^{ss}))\ge 2g(X)-1\, ,  \ {\rm  for} \ n \ge 3  ( \ge g(X) \  {\rm for} \  n=2)\,.$$
(2) Part (2) can be proved similarly. \\
(3) One has $\overline{R}^{'ss} - \overline{R}^{'s} = pr^{-1}(R^{ss} - R^{s})$. Hence  
${\rm codim}_{\overline{R}^{'ss}} (\overline{R}^{'ss} - \overline{R}^{'s})  
 \ge 2g(X) -2 \ {\rm for}\ n \ge 3$ (and $\ge g(X) -1 \ {\rm for}\ n=2)\,.$
\end{proof}

\begin{corollary}\label{Coro1.6}  For $n \ge 3$ (resp. $n = 2$) and
for a general $\overline{L}$, one has:\\
\begin{tabular}{lll}
(1) & {\rm codim}$_{H'}(H' - \overline{H}^{'ss}) \ge 2g(X)-1$  & (resp. $g(X)$) \\
(2) & {\rm codim}$_{H'}(H' - \overline{H}^{'s}) \ge 2g(X)-2$  & (resp. $g(X)-1$) \\
(3) & {\rm codim}$_{\overline{H}^{'ss}}(\overline{H}^{'ss} - \overline{H}^{'s}) \ge 2g(X)-2$  & (resp. $g(X)-1$) 
\end{tabular}
\end{corollary}

Taking GIT quotients by $PGL(N)$, we get the following corollary from Proposition \ref{prop1.5} and Corollary \ref{Coro1.6}.

\begin{corollary}\label{Coro1.7}  For $n \ge 3$ (resp. $n=2$) and $\overline{M}^{'ss}
\neq \overline{M}^{'s}$ we have:\\
\begin{tabular}{lll}
(1) & {\rm codim}$_{M'}(M' - \overline{M}^{'ss}) \ge 2g(X)-1$  & (resp. $g(X)$) \\
(2) & {\rm codim}$_{M'} (M' - \overline{M}^{'s}) \ge 2g(X)-2$  & (resp. $g(X)-1$) \\
(3) & {\rm codim}$_{\overline{M}^{'ss}}(\overline{M}^{'ss} - \overline{M}^{'s}) \ge 2g(X)-2$  & (resp. $g(X)-1$) \\
(4) & {\rm codim}$_{M'_{\overline{L}}}({M}'_{\overline{L}} - \overline{M}^{'ss}_{\overline{L}}) \ge 2g(X) -1$  &  
(resp. $g(X)$).\\
(5) & {\rm codim}$_{M'_{\overline{L}}}({M}'_{\overline{L}} - \overline{M}^{'s}_{\overline{L}}) \ge 2g(X) -2$  &  (resp. $g(X)-1$).\\
(6) & {\rm codim}$_{\overline{M}^{'ss}_{\overline L}}(\overline{M}^{'ss}_{\overline L} - \overline{M}^{'s}_{\overline L}) \ge 2g(X)-2$  & (resp. $g(X)-1$) 
\end{tabular}
\end{corollary}

        Let $\overline{U}^{'ss}_Y(n,d)$ (respectively $\overline{U}^{'s}_Y(n,d)$) be the subset of $U_Y(n,d)$ corresponding to vector bundles $F$ such that $p^*F$ is semistable (respectively stable). 
        
\begin{theorem}\label{thmcodim}  
For $n \ge 3$ (resp. $n=2$) and $\overline{U}^{'ss}_Y(n,d) \neq \overline{U}^{'s}_Y(n,d)$  one has:\\
\begin{tabular}{lll}
(1) & {\rm codim}$_{U'_Y(n,d)} (U'_Y(n,d) - \overline{U}^{'ss}_Y(n,d)) \ge 2g(X)-1$  & (resp. $g(X)$) \\
(2) & {\rm codim}$_{U'_L(n,d)} (U'_L(n,d) - \overline{U}^{'ss}_L(n,d)) \ge 2g(X)-1)$  &  (resp. $g(X)$)\\
(3) & {\rm codim}$_{U'_Y(n,d} (U'_Y(n,d) - \overline{U}^{'s}_Y(n,d)) \ge 2g(X)-2$  & (resp. $g(X)-1$) \\
(4) & {\rm codim}$_{U'_L(n,d)} (U'_L(n,d) - \overline{U}^{'s}_L(n,d)) \ge 2g(X)-2$  &  (resp. $g(X)-1$)\\
(5) & {\rm codim}$_{\overline{U}^{'ss}_Y(n,d)} (\overline{U}^{'ss}_Y(n,d) - \overline{U}^{'s}_Y(n,d)) \ge 2g(X)-2$  & (resp. $g(X)-1$) \\
(6) & {\rm codim}$_{\overline{U}^{'ss}_L(n,d)}  (\overline{U}^{'ss}_L(n,d) - \overline{U}^{'s}_L(n,d)) \ge 2g(X)-2$  & (resp. $g(X)-1$). \\
 
\end{tabular}

\end{theorem}

\begin{proof}
Since $M'$ (respectively $M'_{\overline{L}})$ and $U_Y'(n,d)$ (respectively $U_L'(n,d)$) are isomorphic, the  theorem follows 
from Corollary \ref{Coro1.7}.
\end{proof}

\subsection{The group $\pi_Y$ and bundles associated to its representations}\label{piY} \hfill
        
              Let $Y$ be a complex nodal curve with $g(X) \ge 2$. By tensoring by a line bundle, we  
 normalise $d$ by the condition $-n < d \le 0$. Let $\pi_X$ denote the group (a Fuchsian group) generated by
$2g(X)+1$ generators $a_1, b_1, a_2, b_2, \cdots, a_g, b_g, c$ with only relations 
$$(\Pi_i  \ a_ib_ia_i^{-1}b_i^{-1}) \, c =1\, , \ c^n =1\, .$$
Define 
$$\pi_Y := \pi_X * \mathbb{Z} * \cdots * \mathbb{Z}\, ,$$
with $\mathbb{Z}$ repeated as many times as the number of nodes and $*$ denoting the free product of groups. Let $1_j$ denote the generator of the $j^{th}$ factor $\mathbb{Z}$.

       Let $\zeta$ be a primitive $n$-th root of unity. Let $\tau$ be the character on the cyclic group generated by $\zeta$ defined by $\tau(\zeta) = \zeta^{-d}$. We say that a representation $\rho$ of $\pi_Y$ in ${\rm GL}(n,\mathbb{C})$ is of type $\tau$ if  $\rho(c) = \tau(\zeta) I_n$, where $I_n$ is the identity matrix of rank $n$.  Let $\rho_X= \rho\vert_{\pi_X}\,$. 

         To a representation $\rho$ of $\pi_Y$ in ${\rm GL}(n,\mathbb{C})$ of type $\tau$, we can associate a generalised parabolic vector bundle (GPB) on $X$ and hence a vector bundle $E_{\rho}$ on $Y$ of rank $n$ and degree $d$ (see \cite[Subsection 2.3]{Bh1} for details).
         
\begin{corollary}\label{cororep}
       Let $Y$ be a complex nodal curve with $g(X) \ge 2$. The subset of $U'_Y(n,d)$ (respectively of $U'_L(n,d)$) consisting of vector bundles which come from representations of the group $\pi(Y)$ has complement of codimension at least $2$ except possibly when $n = g(X) = 2, d$ even.
\end{corollary}

\begin{proof}
     By \cite[Theorem 2.8]{Bh1}, there is a bijective correspondence between 
$Rep:= \{$Equivalence  classes  of  representations  $\rho: \pi_Y \to \ {\rm GL}(n, \mathbb{C})$ of  type $ \tau$  such  that $\rho_X$ is  irreducible  and  unitary$\} $ and elements of the open dense subset $\overline{U}^{'s}_Y(n,d)$. This bijection is in fact a homeomorphism as can be seen using the fact that, over a smooth curve, the Narasimhan-Seshadri correspondence is a homeomorphism. Hence the corollary follows from Theorem \ref{thmcodim}.  In the case of $U_L(n,d)$ we take $\rho$ to have values 
in SL$(n, \mathbb{C})$. 
\end{proof}
 
 We note that for $d=0$, the Fuchsian group $\pi_X$ is replaced by the  fundamental group $\pi_1(X)$ of $X$ \cite{Bh}. 

\subsection{Proof of the non-existence theorem} 
   
We start with a few lemmas needed for the proof.

\begin{lemma} \label{L3.1}
Let $E$ (respectively $F$) be a vector bundle of rank $n$, degree $d$ and determinant $L$ (respectively of rank $n+1$, degree $d'$ and 
determinant $L'$) on $Y$. 
Then there exists an integer $m_0= m_0(E,F)$ and an injective homomorphism of vector bundles $i: E \to F(m_0)$.  
One has $F(m_0) / iE \cong \mathcal{O}_Y(n+1) \otimes det F \otimes (det E)^{-1} =  L' \otimes L^{-1}(n+1)$.
\end{lemma} 
\begin{proof}
   This is proved exactly as  \cite[Lemma 3.1]{R}.
\end{proof}

      The tensor product of two semistable vector bundles of rank $\ge 2$ on $Y$ may not be semistable \cite{Bh}. 
 So we may not be able to choose $m_0$ dependent only on the integers $n, d, d'$. 
 If $E$ is (semi)stable, $E\otimes I_y$ may not be (semi)stable,  where $I_y$ is the ideal sheaf at a node $y$.
          However, if a vector bundle $E$ is (semi)stable, then it is easy to see that $E\otimes N$ is (semi)stable for any line bundle $N$ on $Y$. 
If $E$ is (semi)stable,  then the dual bundle $E^*$ is (semi)stable \cite[Lemma 2.6(2)]{UB2}. 

\begin{lemma} \label{L3.2}
    Assume that $E$ and $F$ are as in Lemma \ref{L3.1} and both are stable vector bundles. 
 Let  $\delta_Y$ be the number of nodes of $Y$.
(1)  $H^1(Y, Hom(E, F(m)) = 0$ for $m \ge  m_1$ where $m_1 = d/n - d'/ (n+1) + 2g - 2\, .$\\
(2)  $H^1(Y, Hom(F(m),  L^{-1} \otimes L' (m(n+1)))) = 0$ for $m \ge m_2$ where $m_2 = \frac{d'/(n+1) - d' +d + 2g-2}{n}$.\\
 (3)  The vector bundle $Hom(F(m),  L^{-1} \otimes L' (m(n+1)))$ is generated by sections for 
 $m \ge m_3$, where $m_3 = (1/n)((1+\delta_Y)(n+1) + 2g - 2 + d -d' + d'/(n+1))$.   
\end{lemma}
\begin{proof}
(1) By Serre duality $h^1(Y, Hom(E, F(m))) = h^1(Y, E^* \otimes F(m)) = h^0(Y, Hom (F(m), E\otimes \omega_Y))$. 
Since $F(m)$ and $E\otimes \omega_Y$ are stable, one has $h^0(Y, Hom (F(m), E\otimes \omega_Y))=0$ if 
$\mu(F(m)) \ge \mu(E\otimes \omega_Y)$ i.e., $\mu(F) + m \ge \mu(E)+ 2g-2$ or equivalently, $m \ge \mu(E) - \mu(F) + 2g - 2$.

(2) As in the proof of part (1), $h^1(Y, Hom(F(m),  L^{-1} \otimes L' (m(n+1)))) = 0$ if $d(L^{-1} \otimes L' (m(n+1)) \ge \mu(F)+m +2g-2$ 
i.e., $mn \ge d'/(n+1) - d' +d + 2g-2$.

(3) The vector bundle $Hom(F(m),  L^{-1} \otimes L' (m(n+1)))= F^*\otimes L^{-1}\otimes L'(mn)$ is stable \cite[Lemma 2.6(2)]{UB2}. 
Hence as in \cite[Lemma 5.2']{N}, one sees that $Hom(F(m),  L^{-1} \otimes L' (m(n+1)))$ is generated by global sections if  
$ - \mu(F) + mn +d' -d \ge 2g-2 +(n+1)(1+ \delta_Y)$ i.e. if $m \ge (1/n)((1+\delta_Y)(n+1)+2g-2 +d -d' + d'/(n+1))$. 
\end{proof}

\begin{lemma} \label{L3.3}
Let $E$ and $F$ be as in Lemma \ref{L3.1} and assume that both $p^*(E)$ and $p^*(F)$ are stable vector bundles on $X$. Then: \\
(1) $E^*\otimes F$ is semistable. \\
(2) $H^0(Y, E^*\otimes F(m))$ generates $E^*\otimes F(m)$ for 
$m >  m_4$ where $m_4 = n(n+1)(1+ \delta_Y) + 2g - 2 + d/n - d'/(n+1)$.
\end{lemma}
\begin{proof}
(1)   If $p^*(E)$ and $p^*(F)$ are stable, then $p^*(E)^*$ is stable, and  $p^*(E)^*\otimes p^*(F)$ is semistable as $X$ is a smooth curve. 
Then $p^*(E^*\otimes F) = p^*(E)^*\otimes p^*(F)$ is semistable and therefore $E^* \otimes F$ is semistable \cite[Proposition 3.6]{BS}. \\
(2)    By \cite[Lemma 5.2']{N}, the conclusion of Part (2) holds if one has 
$\mu(E^*\otimes F(m)) > r(E^*\otimes F)(1+ \delta_Y) + 2g -2$ i.e. if $\mu(F) + m - \mu(E) > n(n+1)(1+ \delta_Y) + 2g-2$, hence the result.
\end{proof}

{\bf Assumptions}        Henceforth, we fix an $E \in U'_L(n,d)$ such that $p^*E$ is stable.  We choose an $m > max \{m_1, m_2, m_3, m_4\}$. 
We assume that $(n+1, d') =1$. Then there is a Poincare bundle $\mathcal{U}$ on $U'_{L'}(n+1,d') \times Y$.  
Let $p_1, p_2$ be the projections from $U'_{L'}(n+1,d') \times Y$ to $U'_{L'}(n+1,d')$ and $Y$ respectively. 
We denote by $\overline{U}'_{L'}(n+1,d') \subset U_{L'}(n+1,d')$ the open subvariety of $U_{L'}(n+1,d')$ 
consisting of vector bundles $F$ such that $p^*F$ is stable.

\subsection{The projective space ${\bf P}$ }   \label{P}  \hfill

 Define the projective space ${\bf P}$ by
$${\bf P} := {\bf P}( H^1(Y, Hom(L'\otimes L^{-1}(m(n+1)), E)))\, .$$
Let $H$ denote the hyperplane bundle on ${\bf P}$. 
Let $p_P$ and $p_Y$ denote the projections from ${\bf P} \times Y$ to ${\bf P}$ and $Y$ respectively.
On ${\bf P} \times Y$, there is a family of vector bundles $W$ given by the exact sequence 
\begin{equation}\label{W}
0  \rightarrow p_Y^*E\otimes p_P^*H \rightarrow W \rightarrow p_Y^*( L^{-1} \otimes L' (m(n+1))) \rightarrow  0\, .
\end{equation}

        Let ${\bf P}^s$ denote the open subset of ${\bf P}$ parametrising stable vector bundles. 
For $F \in \overline{U}'_{L'}(n+1,d')$,  $H^0(Y, Hom(E, F(m)))$ generates $Hom(E, F(m))$ at every point $t \in Y$ by Lemma \ref{L3.3}.  
For every $t\in Y$, the set of homomorphisms in Hom $(E_t, F(m)_t)$ which are not injective is an irreducible set of codimension at least $2$. 
The inverse image of this set, under the surjective evaluation map $H^0(Y, Hom(E, F(m))) \to \ {\rm Hom}(E_t, F(m)_t)$ is the set $S_t$ of 
homomorphisms  $E \to  F(m)$ which are not injective at $t$. Hence there is a non-empty open subset of Hom $(E, F(m))$ consisting of injective homomorphisms.  Hence $F(m)$ belongs to the family $W$ parametrised by ${\bf P}^s$. In particular, ${\bf P}^s$ is nonempty. 
By the universal property of moduli spaces, there is a morphism 
$$\lambda: {\bf P}^s \ \rightarrow \  U'_{L'}(n+1,d')\, .$$
One has 
\begin{equation}\label{U} 
(\lambda \times 1_Y)^* \mathcal{U} \cong W \otimes (p_Y)^*(\mathcal{O}_Y(-m)) \otimes p_{P^s}^*(N)
\end{equation}
 for some line bundle $N$ on ${\bf P}^s$ \cite[Lemma 2.5]{R}.  
The restriction of $H$ to ${\bf P}^s$ will be denoted by $H$ again.

\subsection{The projective bundle $P_1$} \label{P1} \hfill
  
Define 
$$V_1 := (p_1)_* (Hom(\mathcal{U} \otimes \mathcal{O}_Y(m), p_2^*(L^{-1} \otimes L' (m(n+1)))))\, , $$
 By Lemma \ref{L3.2}(2), $V_1$ is a vector bundle on $U'_{L'}(n+1,d')$. 
Define   $P_1 := P(V_1)$, a projective bundle on $U'_{L'}(n+1,d')$.   
Let $S_1 \subset P_1$ be the subset corresponding to surjections $F(m) \to L^{-1} \otimes L' (m(n+1))$.  
Let $\pi_1: S_1 \to U'_{L'}(n+1,d')$ be the projection. To give a morphism 
$$\phi: {\bf P}^s \to S_1$$ 
such that 
$\pi_1 \circ \phi = \lambda$, it suffices to give a line subbundle of $\lambda^* V_1$. By equation \eqref{U}, 
$$\lambda^*V_1 \cong  p_{P^s, *} (Hom(W, p_Y^* (L^{-1} \otimes L' (m(n+1))))) \otimes N^{-1}\, .$$
The exact sequence \eqref{W} gives a nowhere vanishing section of $p_{P^s, *} (Hom(W, p_Y^* (L^{-1} \otimes L' (m(n+1)))))$ 
giving a trivial line subbundle of it. Hence $\lambda^*V_1$ contains a line subbundle isomorphic to $N^{-1}$.
It follows that 
$$\phi^* \tau_1 = N$$ 
where $\tau_1$ is the relative hyperplane bundle on $P_1$. 

\subsection{The projective bundle $P_2$} \label{P2}  \hfill

     By Lemma  \ref{L3.2}(1), 
$$V_2 :=  (p_1)_* ( Hom (p_2^*E, \mathcal{U}\otimes p_2^*(\mathcal{O}_Y(m) ) ) )$$   
is a vector bundle on $U'_{L'}(n+1,d')$. Let $P_2 = \bf{P}(V_2)$.
Let $S_2 \subset P_2$ be the subset corresponding to injections $E \to F(m)$.   
Let $\pi_2: S_2 \to U'_{L'}(n+1,d')$ be the projection. We have 
$$\lambda^*V_2 \cong  p_{P^s, *} (Hom(p_Y^*E, W))\otimes N\, .$$
Hence $\lambda^* V_2$ has a line subbundle isomorphic to $H \otimes N$. 
This gives a morphism 
$$\psi: {\bf P}^s \to S_2$$ 
such that $\pi_2 \circ \psi = \lambda$.
Then one has
$$\psi^* \tau_2 = N^{- 1}\otimes H^{-1}$$ 
where $\tau_2$ is the relative hyperplane bundle on $P_2$. 

\begin{lemma} \label{L3.5}
Assume that $g(X) \ge 2$ for $n \ge 3$ and $g(X) \ge 3$ for $n=2$. Then the  subset $D_2 = P_2 - S_2$ has codimension $\ge 2$ in $P_2$ or it is the union of an irreducible divisor and (possibly) a closed subset of codimension $\ge 2$.
\end{lemma}

\begin{proof}
       For $F \in U'_{L'}(n+1,d')$, let $P_{2,F}$ be fibre of $P_2 \to U'_{L'}(n+1,d')$ and $D_{2,F}$ the fibre of $D_2 = P_2 - S_2$ over $F$. Let   $d_F$ denote the constant dimension of $P_{2,F}$. Since  $n+1 \ge 3$ and $(n+1, d')=1$, codim$_{U'_{L'}(n+1,d')} (U'_{L'}(n+1,d') - \overline{U}^{'ss}_{L'}(n+1,d')) \ge 2g(X)-1$ by Theorem \ref{thmcodim}.   It follows that  codim$_{U'_{L'}(n+1,d')} (U'_{L'}(n+1,d') - \overline{U}^{'ss}_{L'}(n+1,d')) \ge 3$ for $g(X) \ge 2$. 
  Hence $pr^{-1}(U'_{L'}(n+1,d') - \overline{U}^{'ss}_{L'}(n+1,d'))$  has codimension $\ge 3$ in $P_2$ and 
 therefore its intersection with $D_2$ has codimension $\ge 2$ in $D_2$. For $F \in \overline{U}^{'ss}_{L'}(n+1,d'))$, by Lemma \ref{L3.3}, Hom$(E, F(m))$ generates $Hom(E_t, F(m)_t)$ for every $t \in Y$. Since the subset of $Hom(E_t, F(m)_t)$ corresponding to non-injective homomorphisms is an irreducible subset of codimension $\ge 2$ (independent of $t$), it follows that the  subset $D_{2,t}$ of Hom$(E, F(m))$ corresponding to homomorphisms which are not injective at $t$ is an irreducible subset of codimension $\ge 2$. Hence $D_{2,F} = \cup_{t\in Y} D_{2,t}$ is an irreducible subset  of codimension $\ge 1$. 
 Since $U'_{L'}(n+1,d')$ is irreducible, so is its open subset $\overline{U}^{'ss}_{L'}(n+1,d')$; hence $\cup_{F \in \overline{U}^{'ss}_{L'}(n+1,d')} D_{2, F}$ is an irreducible  subset of $D_2$ of codimension $\ge 1$. This proves the lemma. 
\end{proof}

\begin{lemma}\label{picPs}
The restriction map
${\rm Pic} \  {\bf P} \to \ {\rm Pic} \ {\bf P}^s$ is an isomorphism so that ${\rm Pic}\ {\bf P}^s \cong \mathbb{Z}$. 
\end{lemma}

\begin{proof}
     Since ${\bf P}$ is smooth, the restriction map ${\rm Pic} \  {\bf P} \to \ {\rm Pic} \ {\bf P}^s$ is surjective. As ${\rm Pic} \  {\bf P} \cong \mathbb{Z}$, it follows that 
$ {\rm Pic} \ {\bf P}^s$ is isomorphic to $\mathbb{Z}$ or $\mathbb{Z}/ c \mathbb{Z}$ for some integer $c\ge 2$, in particular rank ${\rm Pic} \ {\bf P}^s \le 1$. 
For $g(X) \ge 2$, we have ${\rm Pic} \  U'_{L'}(n+1,d') \cong \mathbb{Z}$ \cite[Theorem I]{UB4}. Hence rank ${\rm Pic} \ P_2 = 2$. By Lemma \ref{L3.5}, 
rank ${\rm  Pic} \ S_2 \ge {\rm  rank \ Pic} \ P_2 - 1 = 1$. Since $\psi: {\bf P}^s \to P_2$ maps ${\bf P}^s$ isomorphically onto $S_2$, now it follows that rank 
${\rm Pic} \ {\bf P}^s =1$ and ${\rm Pic} \ {\bf P}^s = \mathbb{Z}$ proving that the restriction map is an isomorphism.       
\end{proof}

\begin{lemma} \label{L3.6}
The image of $\phi^*: {\rm Pic} \  P_1 \to {\rm Pic} \ {\bf P}^s$
is the subgroup of $\mathbb{Z}$ generated by $n$ and $d$. 
\end{lemma}

\begin{proof}
      This can be proved similarly as \cite[Lemma 3.6]{R}.  We briefly sketch it, see \cite[Lemma 3.6]{R} for details. 
  Pic $P_1$ is generated by $\tau_1$ and $\pi^*\theta_{L'}$. 
 Hence the required image is generated by $\phi^*\pi^* \theta_{L'}= \lambda^* \theta_{L'}$ and $\phi^*\tau_1= N$. 
 Let $T_{U_{L'}}$ denote  the tangent bundle of $U'_{L'}(n+1,d')$.  
 By \cite[Theorem 4]{UB3},  $det (T_{U_{L'}})  = \theta_{L'}^2$. By \cite[Remark 2.11]{R}, 
 $\lambda^* \theta_{L'} = H^{2r}$, where $H$ is the restriction of the hyperplane bundle on ${\bf P}^s$ and $r =  (n+1)(nm-d)+nd'$. 
 Since $(n+1, d')=1$, there exists integers $l, e$ with $ld' - e(n+1)= 1$. Taking $N = H^{n'}$, one shows that $n' = l(d - mn) -en$, where $m$ is as chosen in assumptions. It follows that image of $\phi^*: {\rm Pic} \  P_1 \to {\rm Pic} \ {\bf P}^s$ is the subgroup of $\mathbb{Z}$ generated by $r$ and $n'$, which is the same as the subgroup generated by $n$ and $d$. 
\end{proof}

\begin{theorem}\label{nonexistencethm}
   Let $Y$ be an integral nodal curve of geometric genus $g(X) \ge 2$. If $n$ and $d$ are not coprime, then there does not exist a Poincar\'e family on any Zariski open subset of $U_L(n,d)$. 
\end{theorem}

\begin{proof}
     Since $U_L(n,d)$ is irreducible and $U'_L(n,d)$ is an open subset of $U_L(n,d)$, any Zariski open subset of $U_L(n,d)$ 
intersects $U'_L(n,d)$ in a non-empty Zariski open subset $V$. Hence it suffices to show that there does not exist a Poincar\'e family 
on any Zariski open subset $V$ of $U'_L(n,d)$.

       Suppose that there exists a Poincar\'e family $\mathcal{E} \to V \times Y$. Since Hom $(L^{-1}\otimes L'(m(n+1)), \mathcal{E}_v) =0$ 
for all $v \in V$, 
$$\mathcal{V} := R^1(p_V)_ * (Hom (p_Y^*(L^{-1}\otimes L'(m(n+1)))), \mathcal{E}))$$
is a vector bundle on $V$. Let $\mathcal{P} := \bf{P}(\mathcal{V})$.  The projective space ${\bf P}$ defined in subsection \ref{P} 
is the fibre of $\mathcal{P}$ over $E \in V$. The projective bundle $\mathcal{P}$ parametrises a family $\mathcal{W}$ 
of vector bundles of rank $n+1$ and degree $d'$. Let $\mathcal{P}^s$ be the open subset of $\mathcal{P}$ parametrising stable 
vector bundles. Clearly, ${\bf P}^s \subset \mathcal{P}^s$ and the morphism $\phi: {\bf P}^s \to S_1$ extends to   
$\bar{\phi}: \mathcal{P}^s \to S_1$ which is an isomorphism onto an open subset of $S_1$. 

            Since  ${\bf P} = {\bf P}(\mathcal{V})_E$ is a fibre of $\mathcal{P}= {\bf P}(\mathcal{V})$, the restriction map ${\rm Pic} \ {\bf P}(\mathcal{V}) 
 \to {\rm Pic} \  {\bf P}$ is a surjective (with the relative hyperplane bundle $H_V$ on $\mathcal{P}$ mapping to $H$). Using Lemma \ref{picPs}, one gets a surjection  
 ${\rm Pic} \ \mathcal{P}  \to {\rm Pic} \ {\bf P}^s$ which factors through a surjection ${\rm Pic} \ \mathcal{P}^s \to  {\rm Pic} \ {\bf P}^s$.  

    Since $\mathcal{P}^s$ is an open subset of $S_1$ and $S_1$ an open subset of $P_1$,  one  has a composite of surjections 
    $${\rm Pic} \ P_1 \to {\rm Pic} \ S_1 \to {\rm Pic} \ \mathcal{P}^s \to {\rm Pic} \ {\bf P}^s = \mathbb{Z}\, .$$   
  By Lemma \ref{L3.6}, this composite is a surjection if and only if $n$ and $d$ are coprime. Hence the family $\mathcal{E}$ 
  on $V$ exists if and only if $(n,d)=1$.
\end{proof}

\end{document}